\documentclass[12pt,a4paper]{article}
\pdfoutput=1
\usepackage{amsthm,upref}
\usepackage[theoremfont]{newtxtext}
\usepackage{newtxmath}
\usepackage{enumerate}
\usepackage{microtype}
\usepackage{authblk}
\usepackage[margin=3.4cm]{geometry}
\usepackage{hyperref,bookmark}

\title{Semilinear Elliptic Equations on Rough Domains}
\author[1]{Wolfgang Arendt}
\author[2]{Daniel Daners}
\affil[1]{Institut f\"ur Angewandte Analysis, Universit\"at Ulm, Germany\authorcr
  \nolinkurl{wolfgang.arendt@uni-ulm.de}}
\affil[2]{School of Mathematics and Statistics, University of Sydney, NSW 2006, Australia\authorcr
  \nolinkurl{daniel.daners@sydney.edu.au}}
\date{February 6, 2022}

\makeatletter
\hypersetup{pdfauthor={Wolfgang Arendt, Daniel Daners}}
\hypersetup{pdftitle={\@title}}
\makeatother

\numberwithin{equation}{section}
\numberwithin{figure}{section}

\theoremstyle{plain}
\newtheorem{theorem}{Theorem}[section]
\newtheorem{lemma}[theorem]{Lemma}
\newtheorem{proposition}[theorem]{Proposition}
\newtheorem{corollary}[theorem]{Corollary}

\theoremstyle{definition}
\newtheorem{definition}[theorem]{Definition}

\newtheorem{example}[theorem]{Example}

\theoremstyle{remark}
\newtheorem{remark}[theorem]{Remark}

\DeclareMathOperator{\dist}{dist}
\DeclareMathOperator{\Tr}{Tr}
\DeclareMathOperator{\tr}{tr}
\DeclareMathOperator{\supp}{supp}
\DeclareMathOperator{\spr}{r}
\DeclareMathOperator{\spb}{s}
\DeclareMathOperator{\repart}{Re}
\DeclareMathOperator{\aaa}{\mathfrak{a}}
\DeclareMathOperator{\loc}{loc}

\let\oldthebibliography\thebibliography
\renewcommand\thebibliography[1]{
  \oldthebibliography{#1}
  \setlength{\parskip}{.2ex}
  \setlength{\itemsep}{0pt plus 0.3ex}
  \small
}

\begin{document}
\maketitle

\renewcommand{\thefootnote}{}
\footnotetext{\textbf{Mathematics Subject Classification (2020):} 35J61, 47N20, 47H07.}
\footnotetext{\footnotesize\textbf{Keywords:} semilinear elliptic equations; rough domains; abstract logistic equation; Kato inequality; principal eigenvalue}

\begin{abstract}
  The paper makes use of recent results in the theory of Banach lattices and positive operators to deal with abstract semilinear equations. The aim is to work with minimal or no regularity conditions on the boundary of the domains, where the usual arguments based on maximum principles do not apply. A key result is an application of Kato's inequality to prove a comparison theorem for eigenfunctions that only requires interior regularity and avoids the use of the Hopf boundary maximum principle. We demonstrate the theory on an abstract degenerate logistic equation by proving the existence, uniqueness and stability of non-trivial positive solutions. Examples of operators include the Dirichlet Laplacian on arbitrary bounded domains, a simplified construction of the Robin Laplacian on arbitrary domains with boundary of finite measure and general elliptic operators in divergence form.
\end{abstract}

\section{Introduction}
Many arguments used in the theory of semi-linear elliptic boundary value problems rely on the maximum principle, and in particular the Hopf boundary maximum principle as for instance presented in \cite[Theorem~2.7]{protter:67:mpd}. These maximum principles imply positivity and comparison theorems for solutions. Such methods can be recast in functional analytic terms using the theory of positive operators on ordered Banach spaces. That approach was pioneered in Amann's seminal paper \cite{amann:76:fpe}. The theory requires an ordered Banach space whose positive cone has a non-empty interior. In the context of elliptic boundary value problems this forces domains to be regular, most often of class $C^2$ or better such that Hopf's boundary maximum principle applies, see for instance \cite[Section~I.4]{amann:76:fpe}. Interior $C^2$-regularity can be slightly weakened by using Bony's maximum principle, see for instance \cite{bony:67:pme,lions:83:rbm}. Nevertheless, without any workarounds, the requirement that the interior of the positive cone be non-empty excludes the Lebesgue space $L^p$ since the positive cone in $L^p$ has empty interior. This adds a layer of complexity we aim to remove.

A substantial part of this paper is devoted to discuss these abstract tools by making use of more recent developments in functional analysis. The functional analytic tools we rely on are in particular de Pagter's theorem on the positivity of the spectral radius of compact irreducible operators; see \cite{depagter:86:ico,meyer:91:bla} and an abstract version of Kato's inequality from \cite{arendt:84:kic}. The latter replaces arguments usually involving Hopf's boundary maximum principle. We also make use of very recent criteria about positivity improving semigroups from \cite{arendt:20:spp} that replace arguments usually based on maximum principles or the Harnack inequality.

Our main point is to prove the existence of a unique non-trivial positive solution of a semilinear logistic quation with degeneracies in Section~\ref{sec:logistic} under no or very minimal regularity assumptions on $\Omega$. Our most innovative result is a comparison principle for eigenvectors established in Section~\ref{sec:eigenvector-comparison} which allows us to establish a pair of sub-solution and super-solutions in very general situations by local arguments in the interior of the domain; see Proposition~\ref{prop:exist-2}.  In this way we are able to essentially remove all requirements on the regularit of the boundary of $\Omega$ and only work with the local continuity and the global boundedness of solutions rather than relying on the Hopf boundary maximum principle. At the same time we streamline and simplify existing arguments. More precisely, we consider the abstract logistic equation
\begin{equation*}
  Au=\lambda u-m(x)g(x,u)u
\end{equation*}
on $L^p(\Omega)$ ($1\leq p<\infty$), where $\Omega\subseteq\mathbb R^N$ is a bounded domain and $-A$ is the generator of a compact irreducible positive $C_0$-semigroup $(T(t))_{t\geq 0}$. We furthermore assume that the semigroup exhibits some inner regularity and global boundedness properties, that is,
\begin{equation}
  \label{eq:bounded-smoothing}
  T(t)L_p(\Omega)\subseteq BC(\Omega):=C(\Omega)\cap L^\infty(\Omega).
\end{equation}
Such a regularity condition is satisfied under no or rather weak assumptions on $\Omega$ depending on the boundary conditions. It automatically implies the compactness of $T(t)$ as shown in Proposition~\ref{prop:bounded-compact} below. We assume that the nonlinearity $g(x,u)$ is strictly increasing to infinity as $u\to\infty$ and that $m\geq 0$ is a non-zero function possibly having some zero set. This zero set is referred to as a \emph{degeneracy}. For more precise assumptions we refer to Section~\ref{sec:logistic}. We also prove the uniqueness, linear stability and differentiable dependence of the unique positive solution on $\lambda$.

A first tool we establish in Section~\ref{sec:strong-monotonicity} is the strong monotonicity of the spectral radius for pairs of positive operators, one dominating the other. This is usually achieved by the strong maximum principle including the Hopf boundary maximum principle such as it is for instance done in \cite{amann:76:fpe,du:06:ost,lopez:13:lso}. Here we only use the irreducibility and compactness of operators on a Banach lattice. There are simple criteria to test positivity and irreducibility for generators induced by bilinear forms associated with the weak formulation of elliptic equations, see \cite[Theorem~2.6 and~2.10]{ouhabaz:05:ahe}.

In Section~\ref{sec:perturbation-multipliers} we consider eigenvalue perturbations and prove convergence results assuming only weak$^*$ convergence of the sequence of multiplication operators. In Section~\ref{sec:smoothing-semigroups} we introduce locally smoothing and positivity improving semigroups. They reflect the fact that, due to the theory of de Giorgi, Nash and Moser, locally bounded weak solutions to elliptic or parabolic equations are always continuous in the interior of a domain, but that boundary regularity requires smoothness of the boundary. The fact that the semigroups are positivity improving serves as an alternative to the maximum principle. We discuss the Kato inequality in Section~\ref{sec:kato-inequality} and use it as a replacement of the maximum principle to compare eigenvectors associated with operators with different potentials in Section~\ref{sec:eigenvector-comparison}. We finally provide a variety of examples of such semigroups in Sections~\ref{sec:examples-laplacian} and~\ref{sec:elliptic-divergence-form}. This includes a new, more elementary treatment of the Laplace operator with Robin boundary conditions on rough domains in Subsection~\ref{sec:robin-laplacian}.

\section{Strong monotonicity  for principal eigenvalues}
\label{sec:strong-monotonicity}
We will throughout work with the theory of Banach lattices and positive operators as for instance presented in \cite{schaefer:74:blp,meyer:91:bla}.  Let $E$ be a Banach lattice and $S,T\in\mathcal L(E)$ positive operators such that $0\leq S\leq T$. The spectral radius is given by $\spr(S):=\sup\{|\lambda|\colon\lambda\in\sigma(S)\}$. By Hadamard's formula for the spectral radius we have
\begin{equation*}
  \spr(S)
  \leq\lim_{n\to\infty}\|S^n\|^{1/n}
  \leq\lim_{n\to\infty}\|T^n\|^{1/n}
  =\spr(T).
\end{equation*}
The aim of this section is to provide a simple criterion such that $\spr(S)<\spr(T)$.

To be able to formulate the result in a precise manner we need to introduce some terminology. If $E$ is a Banach lattice and $u$ is in the positive cone $E_+$, then we write $u>0$ if $u\geq 0$ and $u\neq 0$. An element $u$ of $E_+$ is called a \emph{quasi-interior point} if $f\land nu\to f$ as $n\to\infty$ for all $f\in E_+$. We write $u\gg 0$ to say that $u\in E_+$ is a quasi-interior point. As an example consider $E=L^p(\Omega)$, where $\Omega\subseteq\mathbb R^N$ is open and $1\leq p\leq\infty$. Then, $u\geq 0$ means that $u(x)\geq 0$ almost everywhere and $u>0$ means that $u\geq 0$ but $u(x)$ does not vanish almost everywhere. If $1\leq p<\infty$, then $u\gg 0$ means that $u(x)>0$ almost everywhere, whereas if $p=\infty$ it means that there exists $\delta>0$ such that $u(x)\geq \delta>0$ almost everywhere.

The dual space $E'$ of $E$ is a Banach lattice as well. Its positive cone $E'_+$ consists of all functionals $\varphi\in E'$ such that $\langle u,\varphi\rangle\geq 0$ for all $u\in E_+$. We say that $\varphi$ is \emph{strictly positive} if $u\geq 0$ and $\langle u,\varphi\rangle=0$ imply that $u=0$. As an example, if $E=L^p(\Omega)$, $1\leq p<\infty$, and $\varphi\in L^{p'}(\Omega)$, then $\varphi\in E_+'$ if and only if $\varphi(x)\geq 0$ almost everywhere, whereas $\varphi$ is strictly positive if and only if $\varphi(x)>0$ almost everywhere.

An operator $0<S\in\mathcal L(E)$ is called \emph{irreducible} if for every closed ideal $J$ of $E$ with $S(J)\subseteq J$ we have $J=\{0\}$ or $J=E$; see \cite[Definition~4.2.1]{meyer:91:bla}. It is easy to see that $S$ is irreducible if and only if for all $0<u\in E$ and $0<\varphi\in E'$ there exists $n\in\mathbb N$ such that $\langle S^nu,\varphi\rangle>0$. It is this latter property we will use here.

We finally call $S\in\mathcal L(E)$ \emph{positivity improving} if $u>0$ implies that $Su\gg 0$ for all $u\in E$. Each positivity improving operator is obviously irreducible.

Now consider a positive irreducible operator $T\in\mathcal L(E)$. If $T^n$ is compact for some $n\in\mathbb N$, then by de Pagter's theorem we have $\spr(T)>0$; see \cite[Theorem~3]{depagter:86:ico} or \cite[Theorem~4.2.2]{meyer:91:bla}. By the Krein-Rutman Theorem there exists a unique $0<u\in E$ such that $\|u\|=1$ and $Tu=\spr(T)u$. We call $u$ the \emph{principal eigenvector} of $T$. In that case $u\gg 0$ and the spectral projection associated with $\{\spr(T\})$ has one-dimensional image. Moreover, $\spr(T)$ is the only eigenvalue of $T$ having a positive eigenvector; see for instance \cite[Theorem~4.1.4 \& 4.2.13]{meyer:91:bla}.

Thus, if $T$ is irreducible, then $T$ has at most one principal eigenvalue and eigenvector. The core of the argument leading to the following theorem appears in \cite{arendt:92:dep}.

\begin{theorem}
  \label{thm:spr-monotone}
  Let $E$ be a Banach lattice and let $S,T\in\mathcal L(E)$ be such that $0\leq S\leq T$. Further suppose that $S$ is irreducible and that $T^k$ is compact for some $k\in\mathbb N$. Then $\spr(S)\leq\spr(T)$ with equality if and only if $S=T$.
\end{theorem}
\begin{proof}
  At the beginning of the section we already saw that $\spr(S)\leq\spr(T)$. Assume now that $\spr(S)=\spr(T)$.  We have to show that $T=S$. Since $T^k$ is compact and $0\leq S\leq T$, it follows from the Theorem of Aliprantis-Burkinshaw that $S^{3k}$ is compact; see \cite[Corollary~3.7.15]{meyer:91:bla}. As $S$ is irreducible it follows from de Pagter's theorem that $\spr(S)>0$ as discussed in the lines preceding the theorem.  Rescaling we can assume that $\spr(S)=\spr(T)=1$. By the Krein-Rutman theorem we can find $0\ll u\in E$ such that $Tu=u$. As $S^{3k}$ is compact, also $(S')^{3k}$ is compact. Again by the Krein-Rutman Theorem there exists $0<\varphi\in E_+$ such that $S'\varphi=\varphi$. We show that $\varphi$ is strictly positive. Let $0<v\in E$. Since $S$ is irreducible, there exists $n\in\mathbb N$ such that $\langle S^nv,\varphi\rangle>0$. Hence
  \begin{equation*}
    \langle v,\varphi\rangle
    =\left\langle v,(S')^n\varphi\right\rangle
    =\left\langle S^nv,\varphi\right\rangle>0,
  \end{equation*}
  showing that $\varphi$ is strictly positive. Next observe that
  \begin{equation*}
    \langle u,T'\varphi-S'\varphi\rangle
    =\langle Tu,\varphi\rangle-\langle u,S'\varphi\rangle
    =\langle u,\varphi\rangle-\langle u,\varphi\rangle
    =0.
  \end{equation*}
  As $u\gg 0$ and $0\leq T'\varphi-S'\varphi$ we conclude that $T'\varphi-S'\varphi=0$. Now let $0\leq f\in E$. Then $Tf-Sf\geq 0$ and
  \begin{equation*}
    \langle Tf-Sf,\varphi\rangle
    =\langle f,T'\varphi-S'\varphi\rangle
    =0.
  \end{equation*}
  As $\varphi$ is strictly positive this implies that $Tf-Sf=0$, that is, $Tf=Sf$ for all $f\in E_+$ and hence for all $f\in E$.
\end{proof}

For later purposes we show that the spectral radius and the principal eigenvector are continuous with respect to $T$. This is really folklore in special cases. We include a proof in our general setting.

\begin{proposition}[Continuity of spectral radius]
  \label{prop:spr-continuous}
  Let $T_n,T\in\mathcal L(E)$ be irreducible positive compact operators such that $T_n\to T$ in $\mathcal L(E)$ as $n\to\infty$. Denote by $u_n$ and $u$ the principal eigenvalues of $T_n$ and $T$ respectively. Then $\spr(T_n)\to\spr(T)$ and $u_n\to u$.
\end{proposition}
\begin{proof}
  By the upper semi-continuity of the spectrum $\limsup_{n\to\infty}\spr(T_n)\leq\spr(T)$; see \cite[Section~IV.3]{kato:76:ptl}. For the opposite inequality note that $\spr(T)$ is an isolated point of $\sigma(T)$, so there exists $\varepsilon_0>0$ such that $B(\spr(T),\varepsilon_0)\cap \sigma(T)=\{\spr(T)\}$. By \cite[Theorem~IV.3.16]{kato:76:ptl}, for every $\varepsilon\in(0,\varepsilon_0)$ there exists $n_0\in\mathbb N$ such that $\sigma(T_n)\cap B(\spr(T),\varepsilon)\neq\emptyset$ for all $n>n_0$. Hence $\liminf_{n\to\infty}\spr(T_n)\geq\spr(T)$ and thus $\lim_{n\to\infty}\spr(T_n)=\spr(T)$.  Let $0<u_n$ be the principal eigenvector of $T_n$. Then
  \begin{equation}
    \label{eq:spr-perturbation}
    Tu_n=(T-T_n)u_n+\spr(T_n)u_n.
  \end{equation}
  Since $\|u_n\|=1$ for all $n\in\mathbb N$ and $T$ is compact there exists a subsequence $(u_{n_k})$ such that $\lim_{k\to\infty}Tu_{n_k}$ exists. Since $\spr(T)>0$, $\spr(T_n)\to\spr(T)$ and $T_n\to T$ we deduce that
  \begin{equation*}
    u:=\lim_{k\to\infty}u_{n_k}
    =\lim_{k\to\infty}\frac{1}{\spr(T_{n_k})}
    \bigl(Tu_{n_k}-(T-T_{n_k})u_{n_k}\bigr)
  \end{equation*}
  exists. As $u_n>0$ and $\|u_n\|=1$ for all $n\in\mathbb N$ we have that $u>0$ and $\|u\|=1$ as well. Replacing $n$ by $n_k$ in \eqref{eq:spr-perturbation} we see that $Tu=\spr(T)u$. Hence $u$ is the principal eigenvector for $T$. By the uniqueness of the limit, $u_n\to u$ as $n\to\infty$.
\end{proof}

Now let $(T(t))_{t\geq 0}$ be a positive $C_0$-semigroup with generator $-A$. We say that $(T(t))_{t\geq 0}$ is \emph{irreducible} if for all $0<f\in E$ and $0<\varphi\in E'$ there exists $t>0$ such that $\langle T(t)u,\varphi\rangle>0$. Let
\begin{equation}
  \label{eq:spbA}
  \spb(-A)
  :=\sup\bigl\{\repart\mu\colon\mu\in\sigma(-A)\bigr\}
  \in[-\infty,\infty)
\end{equation}
be the \emph{spectral bound} of $-A$. Then
\begin{equation*}
  \lambda_1(A):=\inf\bigl\{\repart\mu\colon\mu\in\sigma(A)\bigr\}
  =-\spb(-A)
\end{equation*}
and $(\mu I+A)$ is invertible and $(\mu I+A)^{-1})\geq 0$ for all $\mu>-\lambda_1(A)$. The semigroup is irreducible if and only if $(\mu I+A)^{-1}$ is positivity improving for some (equivalently all) $\mu>-\lambda_1(A)$. In particular $(\mu I+A)^{-1}$ is irreducible. Note that
\begin{equation}
  \label{eq:spr-spb}
  \spr\bigl((\mu I+A)^{-1}\bigr)
  =\frac{1}{\mu+\lambda_1(A)}
\end{equation}
for all $\mu>-\lambda_1(A)$. Now assume that $(T(t))_{t\geq 0}$ is irreducible and that $(\mu I+A)^{-1}$ is compact for one (equivalently all) $\mu>-\lambda_1(A)$. Then by \eqref{eq:spr-spb} and de Pagter's Theorem $\lambda_1(A)<\infty$, that is, $\sigma(A)\neq\emptyset$, and $\lambda_1(A)$ is the smallest eigenvalue of $A$. For $u\in E$ we have $(\mu I+A)^{-1}u=\spr\bigl((\mu I+A)^{-1}\bigr)u$ if and only if $u\in D(A)$ and $Au=\lambda_1(A)u$. Thus the principal eigenvector $u$ of $(\mu I+A)^{-1}$ is also the unique $0<u\in D(A)$ such that $Au=\lambda_1(A)u$, $\|u\|=1$. For this reason we call $u$ also the \emph{principal eigenvector of $A$}. In fact, $\lambda_1(A)$ is the only eigenvalue of $A$ having a positive eigenvector. Moreover, $u\gg 0$ as we have seen before.

\begin{proposition}
  \label{prop:A-pev}
  Let $(T(t))_{t\geq 0}$ be a positive irreducible $C_0$-semigroup with generator $-A$. Assume that $A$ has compact resolvent. If $0<u\in D(A)$, $\lambda\in\mathbb R$ and $Au=\lambda u$, then $\lambda=\lambda_1(A)$ and $u/\|u\|$ is the principal eigenvector of $A$. Moreover, if the semigroup is also holomorphic, then there exists $\delta>0$ such that $\repart\lambda\geq\lambda_1(A)+\delta$ for all $\lambda\in\sigma(A)\setminus\{\lambda_1(A)\}$.
\end{proposition}
\begin{proof}
  By a translation and a rescaling we may assume that $\lambda_1(A)=0$ and thus ${\spr\bigl((I+A)^{-1}\bigr)}=1$. By the Krein-Rutman Theorem there exists $0<\varphi\in D(A')$ such that $A'\varphi=0$. Since $(I+A)^{-1}$ is positivity improving $\varphi$ is strictly positive, see the proof of Theorem~\ref{thm:spr-monotone}. Now let $0<u\in D(A)$, $\lambda\in\mathbb R$ such that $Au=\lambda u$. Then
  \begin{equation*}
    \lambda\langle u,\varphi\rangle
    =\langle Au,\varphi\rangle
    =\langle u,A'\varphi\rangle
    =0.
  \end{equation*}
  As $\langle u,\varphi\rangle>0$ it follows that $\lambda=0=\lambda_1(A)$. Since holomorphic semigroups are norm-continuous the last assertion on the strict dominance follows from \cite[Corollary~C-III.3.17]{nagel:86:osp}.
\end{proof}

\section{Perturbation by multiplication operators}
\label{sec:perturbation-multipliers}
Let $\Omega$ be an open set in $\mathbb R^N$. We assume that $E=L^p(\Omega)$ for some $1\leq p<\infty$. We investigate the effect of a perturbation of the generator of a $C_0$-semigroup on $E$ by a bounded multiplication operator. We start by proving a strong comparison result for principal eigenvalues analogous to the classical one that is often obtained using the maximum principle, see for instance \cite[Lemma~5.2]{du:06:ost} or \cite[Proposition~8.3]{lopez:13:lso}, but here we use abstract methods to allow very weak regularity assumptions.

\begin{proposition}
  \label{prop:ev-monotone}
  Let $\Omega\subseteq\mathbb R^N$ be an open set and let $m,m_1,m_2\in L^\infty(\Omega)$.  Let $(T(t))_{t\geq 0}$ be a positive $C_0$-semigroup on $E$ with generator $-A$.  Then the following assertions are true:
  \begin{enumerate}[\upshape (i)]
  \item The operator $-(A+m)$ generates a positive $C_0$-semigroup $(T_m(t))_{t\geq 0}$ on $E$ satisfying
    \begin{equation}
      \label{eq:semigroup-bounds}
      e^{-\omega t}T(t)\leq T_m(t)\leq e^{\omega t}T(t)
    \end{equation}
    for all $t\geq 0$, where $\omega:=\|m\|_\infty$. Moreover,
    \begin{equation}
      \label{eq:semigroup-comparison}
      m_2\leq m_1\implies T_{m_1}(t)\leq T_{m_2}(t)
      \quad\text{for all $t\geq 0$}
    \end{equation}
  \item If $(T(t))_{t\geq 0}$ is irreducible and $A$ has compact resolvent, then $(T_m(t))_{t\geq 0}$ is irreducible and $-(A+m)$ has compact resolvent. Moreover, $\lambda_1(m):=\lambda_1(A+m)<\infty$ and
    \begin{equation}
      \label{eq:eigenvalue-comparison}
      m_1\leq m_2\implies \lambda_1(m_1)\leq \lambda_1(m_2)
    \end{equation}
    with equality if and only if $m_1=m_2$ almost everywhere.
  \end{enumerate}
\end{proposition}
\begin{proof}
  Assertions \eqref{eq:semigroup-bounds} and \eqref{eq:semigroup-comparison} follow from Trotter's formula
  \begin{equation*}
    T_m(t)f=\lim_{n\to\infty}\left(T(t/n)e^{-mt/n}\right)^nf,
  \end{equation*}
  see for instance \cite[Corollary~III.5.8]{engel:00:ops}.  By means of the Laplace transform representation of the resolvent of a generator we deduce from \eqref{eq:semigroup-comparison} that
  \begin{equation*}
    0\leq (\mu I+A+m_2)^{-1}\leq(\mu I+A+m_1)^{-1}
  \end{equation*}
  if $\mu$ is large enough. In particular this implies that
  \begin{equation}
    \label{eq:spr-comparison}
    \spr((\mu I+A+m_2)^{-1})\leq\spr((\mu I+A+m_1)^{-1}).
  \end{equation}
  If we assume that $(T(t))_{t\geq 0}$ is irreducible and that $A$ has compact resolvent, then \eqref{eq:semigroup-bounds} implies that $(T_m(t))_{t\geq 0}$ is irreducible. Also, for large enough $\mu$ we have
  \begin{equation*}
    (\mu I+A+m)^{-1}=\bigl(I-m(\mu I+A+m)^{-1}\bigr)(\mu I+A)^{-1}
  \end{equation*}
  and therefore $A+m$ has compact resolvent. Formula \eqref{eq:spr-spb} shows that \eqref{eq:eigenvalue-comparison} and \eqref{eq:spr-comparison} are equivalent. Since $T_m$ is irreducible, $(\mu I+A+m)^{-1}$ is positivity improving and hence irreducible. If equality holds in \eqref{eq:eigenvalue-comparison} and hence in \eqref{eq:spr-comparison}, then Theorem~\ref{thm:spr-monotone} implies that $(\mu I+A+m_2)^{-1} =(\mu I+A+m_1)^{-1}$ and hence $A+m_1=A+m_2$, that is, $m_1=m_2$ almost everywhere.
\end{proof}

The above proposition tells us that the function $m\mapsto \lambda_1(m)$ is increasing as a function of $m$. We want to show that it is also continuous. The following result seems to be new under our weak hypotheses. We will only assume weak$^*$-convergence of $m$, not uniform convergence such as in \cite[Lemma~5.2]{du:06:ost} and \cite[Corollary~8.1]{lopez:13:lso} or monotonicity as in \cite[Lemma~2.1]{daners:18:gdg}.

\begin{theorem}
  \label{thm:ev-continuous}
  Let $E=L^p(\Omega)$, where $1\leq p<\infty$ and $\Omega\subseteq\mathbb R^N$ is an open set. Assume that $-A$ generates a positive, irreducible $C_0$-semigroup $(T(t))_{t\geq 0}$ on $E$ and that $A$ has compact resolvent.  Suppose that $0\leq m_n,m\in L^\infty(\Omega)$ such that $m_n\stackrel{w^*}{\rightharpoonup} m$ weak$^*$ in $L^\infty(\Omega)$. Then the following assertions are true:
  \begin{enumerate}[\upshape (i)]
  \item If $\omega\in\varrho(-(A+m))$, then there exists $n_0\in\mathbb N$ such that $\omega\in\varrho(-(A+m_n))$ for all $n\geq n_0$ and $(\omega I+ A+m_n)^{-1}\to (\omega I+A+m)^{-1}$ in $\mathcal L(E)$.
  \item $\lambda_1(m_n)\to\lambda_1(m)$ and if $u_n$ and $u$ are the
    corresponding principal eigenvectors, then $u_n\to u$ in $E$.
  \end{enumerate}
\end{theorem}
\begin{proof}
  Since $m_n$ is weak$^*$-convergent in $L^\infty$ it follows that the sequence is bounded in $L^\infty$. By \eqref{eq:semigroup-bounds} there exists $M\geq 1$ and $\omega_0\in\mathbb R$ such that
  \begin{equation*}
    \|T_{m_n}(t)\|\leq Me^{\omega_0t}
    \qquad\text{and}\qquad
    \|T(t)\|\leq Me^{\omega_0t}
  \end{equation*}
  for all $t\geq 0$. Fix $\omega\geq\omega_0+1$. Then by the Laplace transformation representation of the resolvent
  \begin{equation*}
    \|(\omega I+A+m_n)^{-1}\|
    =\Bigl\|\int_0^\infty e^{-\omega t}T_{m_n}(t)\,dt\Bigr\|
    \leq M\int_0^\infty e^{-(\omega_0+1)t}e^{\omega_0t}\,dt
    =M
  \end{equation*}
  for all $n\in\mathbb N$. Keeping the same $\omega$ as before we set $R_n:=(\omega I+A+m_n)^{-1}$ and $R:=(\omega I+A+m)^{-1}$. We show that $R_n\to R$. To see this note that
  \begin{equation*}
    R_n-R
    =R_n(m_n-m)R
    =R_nC_nR
  \end{equation*}
  if we define $C_n\in\mathcal L(E)$ by $C_nf:=(m_n-m)f$ for all $f\in E$. Hence
  \begin{equation*}
    R_n=R+R_nC_nR
    =R+(R+R_nC_nR)C_nR
    =R+RC_nR+R_nC_n(RC_nR).
  \end{equation*}
  As $R_nC_n$ is uniformly bounded it is sufficient to show that
  \begin{equation}
    \label{eq:rcr-zero}
    \lim_{n\to\infty}\|RC_nR\|_{\mathcal L(E)}=0.
  \end{equation}
  As $m_n-m\stackrel{w^*}{\rightharpoonup} 0$ in $L^\infty$ we have $\langle g,C_nf\rangle\to 0$ for all $f\in L^p(\Omega)$ and $g\in E'=L^{p'}(\Omega)$, where $1/p+1/p'=1$ (with $p'=\infty$ if $p=1$). Since $R$ is compact, also $R'$ is compact. Denote by $B$ and $B'$ the unit balls in $E$ and $E'$. Then, $K:=\overline{RB}$ and $K':=\overline{R'B'}$ are compact in $E$ and $E'$ respectively. It follows from the equi-continuity of the family $(C_n)$ that for $f\in B$,
  \begin{equation*}
    \|RC_nRf\|
    =\sup_{g\in B'}|\langle C_nRf,R'g\rangle|
    \leq\sup_{u\in K, v\in K'}|\langle C_nu,v\rangle|\to 0
  \end{equation*}
  since $K$ and $K'$ are compact. Hence, \eqref{eq:rcr-zero} follows. Taking into account \cite[Theorem~IV.2.25]{kato:76:ptl} this proves (i). Now (ii) follows from Proposition~\ref{prop:spr-continuous}.
\end{proof}

We note that the above theorem in fact implies that $u_n\to u$ in $D(A^k)$ for all $k\in\mathbb N$, where we take the graph norm on $D(A^k)$. Indeed, $A^ku_n=\lambda_1(m_n)^ku_n\to\lambda_1(m)^ku=A^ku$ in $E$ as $n\to\infty$ for all $k\in\mathbb N$. In many applications to boundary value problems on sufficiently smooth domains $D(A)$ is a closed subspace of $W^{2,p}(\Omega)$, and hence the principal eigenvectors converge in $W^{2.p}(\Omega)$, not just in $L^p(\Omega)$.

\section{Locally smoothing semigroups}
\label{sec:smoothing-semigroups}
Throughout this section we let $\Omega\subseteq\mathbb R^N$ be an open set, $1\leq p<\infty$ and $E:=L^p(\Omega)$.  We will impose the following regularity property. That property is motivated by the general principle in the theory of parabolic partial differential equations that solutions exhibit local regularity properties, but not necessarily up to the boundary.

\begin{definition}[Locally smoothing semigroup]
  \label{def:locally-smoothing}
  A $C_0$-semigroup $(T(t))_{t\geq 0}$ on $E$ is called \emph{locally smoothing} if
  \begin{equation}
    \label{eq:smoothing-semigroup}
    T(t)E\subseteq C(\Omega)\quad\text{for all $t>0$}
  \end{equation}
  and for every $x\in\Omega$ there exist $w\in E$ and $s>0$ such that
  \begin{equation}
    \label{eq:zero-regular}
    [T(s)w](x)\neq 0.
  \end{equation}
\end{definition}

We note that if $(T(t))_{t\geq 0}$ is a positive semigroup, then we can choose the function $w$ in \eqref{eq:zero-regular} such that $w>0$ by replacing $w$ by $|w|$. We next show that the local smoothing property and ultra-contractivity of a semigroup is preserved under a perturbation by a bounded measurable potential.

\begin{remark}
  Condition \eqref{eq:zero-regular} is a kind of continuity property at $t=0$. Examples of locally smoothing semigroups $(T(t))_{t\geq 0}$ (for example on Lipschitz domains) occur if one of the following conditions are satisfied, see \cite{arendt:20:spp}:

  \begin{enumerate}[(a)]
  \item $\Omega$ is bounded, $T(t)E\subseteq C(\bar\Omega)$
    and $(T(t)|_{C(\bar\Omega)})_{t\geq 0}$ is a $C_0$-semigroup.

  \item $\Omega$ is bounded, $T(t)E\subseteq C_0(\Omega)$
    and $(T(t)|_{C_0(\Omega)})_{t\geq 0}$ is a $C_0$-semigroup. Here,
    \begin{equation}
      \label{eq:C0-space}
      C_0(\Omega)=\{u\in C(\bar\Omega)\colon u|_{\partial\Omega}=0\}.
    \end{equation}
  \end{enumerate}
  In Sections~\ref{sec:examples-laplacian} and~\ref{sec:elliptic-divergence-form} we will encounter examples of locally smoothing semigroups not satisfying (a) or (b).
\end{remark}

For a semigroup to satisfy \eqref{eq:bounded-smoothing} we also recall the following standard definition.

\begin{definition}[Ultra-contractive semigroup]
  \label{def:ultra-contractive}
  A $C_0$-semigroup $(T(t))_{t\geq 0}$ is called \emph{ultra-contractive} if $T(t)E\subseteq L^\infty(\Omega)$ for all $t>0$.
\end{definition}

We refer to \cite[Chapter~2]{davies:89:hks} or \cite[Section~7.3]{arendt:04:see} for further properties associated with ultra-contractive semigroups. For us the following two properties are useful.

\begin{theorem}
  \label{thm:m-locally-smoothing}
  Let $-A$ be the generator of a positive $C_0$-semigroup $(T(t))_{t\geq 0}$ on $E$ and suppose that $m\in L^\infty(\Omega)$. Then the following assertions hold.
  \begin{enumerate}[\upshape (i)]
  \item If $(T(t))_{t\geq 0}$ is ultra-contractive, then also $(T_m(t))_{t\geq 0}$ is ultra-contractive.
  \item If $(T(t))_{t\geq 0}$ is locally smoothing, then also $(T_m(t))_{t\geq 0}$ is locally smoothing.
  \end{enumerate}
\end{theorem}
\begin{proof}
  (i) This immediately follows from \eqref{eq:semigroup-bounds}.

  (ii) We know from Proposition~\ref{prop:ev-monotone} that $-(A+m)$ generates a positive $C_0$-semigroup on $E$. We also have
  \begin{equation}
    \label{eq:var-constants}
    T_m(t)u=T(t)u+\int_0^tT(t-s)\bigl(mT_m(s)u\bigr)\,ds
  \end{equation}
  for all $u\in E$; see \cite[Corollary~III.1.7]{engel:00:ops}. The integral is defined as a Bochner- (or Riemann-) integral with values in $E$. By assumption $T(t)u\in C(\Omega)$. Hence, in order to show that $T_m(t)u\in C(\Omega)$ it is sufficient to show that the integral on the right hand side of \eqref{eq:var-constants} is in $C(\Omega)$. We note that a function is continuous on $\Omega$ if and only if it is continuous on every bounded open set $\Omega_0\subseteq\Omega$ with $\bar\Omega_0\subseteq\Omega$. Let $\Omega_0$ be such a set. As the restriction operator $r_{\bar\Omega_0}(v):= v|_{\bar\Omega_0}$ is bounded and linear from $L_p(\Omega)$ to $L_p(\Omega_0)$ we have
  \begin{equation*}
    r_{\bar\Omega_0}\Bigl(\int_0^tT(t-s)\bigl(mT_m(s)u\bigr)\,ds\Bigr)
    =\int_0^tr_{\bar\Omega_0}\bigl(T(t-s)\bigl(mT_m(s)u\bigr)\bigr)\,ds;
  \end{equation*}
  as Bochner integrals; see \cite[Proposition~1.1.6]{arendt:11:vlt}. Hence it is sufficient to show that the right hand side also exists as a Bochner integral in $C(\bar\Omega_0)$. Define $F\colon (0,t)\to C(\bar\Omega_0)$ by
  \begin{equation*}
    F(s):=r_{\bar\Omega_0}\bigl(T(t-s)\bigl(mT_m(s)u\bigr)\bigr).
  \end{equation*}
  To see that $F$ is Bochner integrable note that if $g\in L_{p'}(\bar\Omega_0)=\bigl(L_p(\bar\Omega_0)\bigr)'$, then $s\to\langle F(s),g\rangle$ is measurable. Hence $\langle F(\cdot),g\rangle$ is measurable for all $g$ in a separating subspace of $C(\bar\Omega_0)'$.  We also note that $C(\bar\Omega_0)$ is separable. Hence, by a generalised version of Pettis' theorem this implies that $F$ is measurable; see \cite[Corollary~1.1.3]{arendt:11:vlt}. In particular, $s\mapsto \|F(s)|_{\bar\Omega_0}\|_\infty$ is measurable and
  \begin{multline}
    \label{eq:pointwise-bound}
    |T_m(t-s)(mT(s)u)|
    \leq T_m(t-s)\bigl(|m|T(s)|u|\bigr)\\
    \leq e^{\omega(t-s)}T(t-s)\bigl(\|m\|_\infty T(s)|u|\bigr)
    =e^{\omega t}\|m\|_\infty T(t)|u|.
  \end{multline}
  Thus $\|F(s)\|_\infty\leq e^{\omega t}\|m\|_\infty\|T(t)|u|\|_\infty$ for all $s\in(0,t)$, where the supremum norm is taken over $\bar\Omega_0$. Therefore, $F\in L^1\bigl((0,t),C(\bar\Omega_0)\bigr)$. This shows that $r_{\bar\Omega_0}\bigl(T_m(t)E\bigr)\subseteq C(\bar\Omega_0)$ for all $t>0$. As the argument works for every choice of $\Omega_0$ we deduce that $T_m(t)u\in C(\Omega)$.

  For the last claim note that by \eqref{eq:semigroup-bounds} we have $T_m(s)w(x)\geq e^{-\omega s}T(s)w(x)>0$ if we choose $w>0$ as in \eqref{eq:zero-regular}. Hence $(T_m(t))_{t\geq 0}$ is locally smoothing.
\end{proof}

\begin{remark}
  There are situations, where $T(t)E\subseteq C(\bar\Omega)$. In that case, in the proof of the above theorem we can choose $\Omega_0=\Omega$ and conclude that $T_m(t)E\subseteq C(\bar\Omega)$.
\end{remark}

We next show that any ultra-contractive semigroup on a bounded domain is compact.

\begin{proposition}
  \label{prop:bounded-compact}
  Let $1\leq p<\infty$ and suppose $\Omega\subseteq\mathbb R^N$ is bounded. Suppose that $-A$ is the generator of an ultra-contractive $C_0$-semigroup $(T(t))_{t>0}$ on $L^p(\Omega)$. Then $A$ has compact resolvent.
\end{proposition}
\begin{proof}
  Set $E:=L^p(\Omega)$ and let $u:=\boldsymbol 1$ be the constant function with value one. As $\Omega$ has finite measure the principal ideal $E_u$ in $E$ generated by $u$ is given by $E_u=L^\infty(\Omega)$. By assumption $T(t/2)E\subseteq E_u$ for all $t>0$. Hence \cite[Theorem~2.2]{daners:17:rds} implies that $T(t)=T(t/2)T(t/2)$ is compact for all $t>0$. In particular $A$ has compact resolvent.
\end{proof}

Suppose that $(T(t))_{t\geq 0}$ is a compact irreducible $C_0$-semigroup. Let $u$ be the principal eigenvector of $A$ corresponding to the principal eigenvalue $\lambda_1(A)$. Then
\begin{equation}
  \label{eq:eigenvector-semigroup}
  T(t)u=e^{-\lambda_1(A)t}u\quad\text{for all $t>0$.}
\end{equation}
If the semigroup is locally smoothing, then $u\in C(\Omega)$. Moreover, if $T(t)$ is irreducible on $E=L^p(\Omega)$ it follows that $u$ is a quasi-interior point of $L^p(\Omega)$, which means that $u(x)>0$ for almost every $x\in\Omega$. We have in fact a stronger statement.

\begin{theorem}[positivity improving semigroups]
  \label{thm:positivity-improving}
  Suppose that $-A$ is the generator of positive irreducible holomorphic $C_0$-semigroup $(T(t))_{t\geq 0}$ on $E=L^p(\Omega)$, where $1\leq p<\infty$. Assume that the semigroup is ultra-contractive and locally smoothing. Then $T(t)$ is compact for all $t>0$ and the following assertions are true.
  \begin{enumerate}[\upshape (a)]
  \item If $u_0\in E$ with $u_0> 0$, then $u(t)=T(t)u_0\in BC(\Omega)$ and $u(t)(x)>0$ for all $t>0$ and all $x\in\Omega$.
  \item If $u$ is the principal eigenvector of $A$, then $u\in BC(\Omega)$ and $u(x)>0$ for all $x\in\Omega$.
  \end{enumerate}
\end{theorem}
\begin{proof}
  By Proposition~\ref{prop:bounded-compact} the semigroup $T(t)$ is compact and thus $A$ has compact resolvent. Since the semigroup is positive and irreducible there exists a principal eigenvalue and eigenvector.  We apply \cite[Theorem~3.1]{arendt:20:spp}, noting that conditions (I) and (II) are given by assumption. This proves (a) and thus (b) by \eqref{eq:eigenvector-semigroup}.
\end{proof}

\section{Positive semigroups and Kato's inequality}
\label{sec:kato-inequality}
One of our main tools to establish an eigenvector comparison for different potentials is an abstract version of Kato's inequality closely related to that from \cite{arendt:84:kic}. For completeness we prove the exact version we need in this paper.

\begin{theorem}[Kato inequality]
  \label{thm:kato-ineq}
  Let $\Omega\subseteq\mathbb R^N$ be a bounded domain and let $E$ be one of the Banach lattices $L^p(\Omega)$ with $1\leq p<\infty$, $C_0(\Omega)$ or $C(\bar\Omega)$. Assume that $-A$ is the generator of a positive $C_0$-semigroup $(T(t))_{t\geq 0}$ on $E$. Then
  \begin{equation}
    \label{eq:kato-ineq}
    \langle 1_{\{u>0\}}Au,\varphi\rangle
    \geq\langle u^+,A'\varphi\rangle
  \end{equation}
  for all $u\in D(A)$ and all $0\leq\varphi\in D(A')$.
\end{theorem}
\begin{proof}
  Let $u\in D(A)$. By the positivity of the semigroup generated by $-A$ and since $u\leq u^+$ we have $1_{\{u>0\}}T(t)u\leq 1_{\{u>0\}}T(t)u^+\leq T(t)u^+$. We also have $u^+=1_{\{u>0\}}u$. Hence, as $u\in D(A)$
  \begin{equation*}
    1_{\{u>0\}}Au
    =1_{\{u>0\}}\lim_{t\downarrow 0}\frac{u-T(t)u}{t}
    \geq\lim_{t\downarrow 0}\frac{u^+-T(t)u^+}{t}.
  \end{equation*}
  If $E=L^p(\Omega)$ with $1\leq p<\infty$ then $1_{\{u>0\}}T(t)u^+\in E$. Hence if $0\leq\varphi\in D(A')$, then
  \begin{equation*}
    \bigl\langle 1_{\{u>0\}}Au,\varphi\bigr\rangle
    \geq\lim_{t\downarrow 0}
    \Bigl\langle\frac{u^+-T(t)u^+}{t},\varphi\Bigr\rangle
    =\lim_{t\downarrow 0}
    \Bigl\langle u^+, \frac{\varphi-T(t)'\varphi}{t}\Bigr\rangle
    =\langle u^+,A'\varphi\rangle
  \end{equation*}
  since $A'$ is the the weak$^*$ generator of $T(t)'$; see for instance \cite[Section~I-A.3.4]{nagel:86:osp}.  If $E=C_0(\Omega)$ or $E=C(\bar\Omega)$, then for $u\in D(A)$ we have
  \begin{equation*}
    Au=\lim_{t\downarrow 0}\frac{u-T(t)u}{t}
  \end{equation*}
  uniformly. By the Riesz representation theorem as found for instance in \cite[Theorem~6.19]{rudin:74:rca} every bounded linear functional on $E$ is represented by a regular Borel measure $\omega$ on $\Omega$ or $\bar\Omega$, respectively. Hence, assuming that $\omega\in D(A')$,
  \begin{multline*}
    \bigl\langle 1_{\{u>0\}}Au,\omega\bigr\rangle
    =\int_{\{u>0\}}\lim_{t\downarrow 0}\frac{u-T(t)u}{t}\,d\omega
    =\lim_{t\downarrow 0}\int_{\{u>0\}}\frac{u-T(t)u}{t}\,d\omega\\
    \geq\lim_{t\downarrow 0}
    \int_{\{u>0\}}\frac{u^+-T(t)u^+}{t}\,d\omega
    =\lim_{t\downarrow 0}
    \Bigl\langle\frac{\omega-T(t)'\omega}{t},u^+\Bigr\rangle
    =\langle u^+,A'\omega\rangle,
  \end{multline*}
  where the interchange of limit and integration is justified by the uniform convergence of the integrand. The last equality is again because $A'$ is the weak$^*$ generator of $T(t)'$
\end{proof}

\section{Comparison of eigenvectors}
\label{sec:eigenvector-comparison}
In this section we establish a comparison principle for eigenvectors that is the key for our application to the logistic equation with minimal regularity. In particular it is designed for situations where boundary regularity is absent and where the Hopf boundary maximum principle is not applicable.  Througout this section, let $\Omega\subseteq\mathbb R^N$ be open and assume that $E=L^p(\Omega)$, $1\leq p<\infty$.

\begin{theorem}[Eigenvector comparison]
  \label{thm:ev-comparison}
  Suppose that $0<m\in L^\infty(\Omega)$ has compact support in $\Omega$. Furthermore assume that $-A$ generates a positive, irreducible and locally smoothing semigroup $(T(t))_{t\geq 0}$ on $E$ and has compact resolvent. Then also $-(A+m)$ generates a positive, irreducible and locally smoothing semigroup and has compact resolvent.

  Moreover, the principal eigenvectors $u_0$ and $u_m$ of $A$ and $A+m$ are in $E$ and there exists a constant $c>0$ such that
  \begin{equation}
    \label{eq:ev-comparison}
    u_0(x)\leq cu_m(x)
  \end{equation}
  for all $x\in\Omega$.
\end{theorem}
\begin{proof}
  The fact that $-(A+m)$ generates a positive irreducible locally smoothing $C_0$-semigroup follows from Theorem~\ref{thm:m-locally-smoothing}.  It has compact resolvent by Proposition~\ref{prop:ev-monotone}. By replacing $A$ by $\omega I+A$ for $\omega$ large enough we can assume without loss of generality that $A^{-1}$ exists and is positive. In particular the principal eigenvalues $\lambda_1(0)$ and $\lambda_1(m)$ of $A$ and $A+m$ are both positive. Denote the corresponding principal eigenvectors by $u_0$ and $u_m$, respectively. Subtracting $c>0$ times the equation $(A+m)u_m=\lambda_1(m)u_m$ from the equation $Au_0=\lambda_1(0)u_0$ we deduce that
  \begin{equation*}
    A(u_0-cu_m)-mcu_m=\lambda_1(0)u_0-\lambda_1(m)cu_m.
  \end{equation*}
  By Proposition~\ref{prop:ev-monotone}, we see that $\lambda_1(0)\leq \lambda_1(m)$ and so
  \begin{equation*}
    A(u_0-cu_m)\leq mcu_m+\lambda_1(0)(u_0-cu_m)
  \end{equation*}
  for all $c\geq 0$.  By assumption $\supp(m)\subseteq\Omega$ is compact. As $u_m\in C(\Omega)$ there exists $\delta>0$ such that $u_m(x)>\delta$ for all $x\in\supp(m)$.  Hence there exists $c_0>0$ such that $m(u_0-cu_m)\leq 0$ for all $c\geq c_0$.  Setting
  \begin{equation*}
    v_c:=u_0-cu_m.
  \end{equation*}
  we have $m1_{\{v_c>0\}}=0$. Thus we are left with
  \begin{equation}
    \label{eq:vc-w-ineq}
    1_{\{v_c>0\}}Av_c\leq \lambda_1(0)v_c^+.
  \end{equation}
  for all $c\geq c_0$.  Now by Kato's inequality from Theorem~\ref{thm:kato-ineq}
  \begin{equation*}
    \langle 1_{\{v_c>0\}}A v_c,\psi\rangle
    \geq\langle v_c^+,A'\psi\rangle
  \end{equation*}
  for all $0\leq\psi\in D(A')$. Combining this with \eqref{eq:vc-w-ineq} we see that
  \begin{equation*}
    \langle v_c^+,A'\psi\rangle
    \leq\langle 1_{\{v_c>0\}}A v_c,\psi\rangle
    \leq\lambda_1(0)\langle v_c^+,\psi\rangle
  \end{equation*}
  for all $0\leq\psi\in D(A')$ and all $c\geq c_0$. Choosing $0\leq\varphi\in E_+'$ and setting $\psi:=(A^{-1})'\varphi$ we see that
  \begin{equation*}
    \langle v_c^+,\varphi\rangle
    \leq\lambda_1(0)\langle A^{-1}v_c^+,\varphi\rangle
  \end{equation*}
  for all $c\geq c_0$. As $0\leq\varphi\in E'$ is arbitrary we deduce that
  \begin{equation}
    \label{eq:vc-ineq}
    0\leq v_c^+\leq\lambda_1(0)A^{-1}v_c^+.
  \end{equation}
  for all $c\geq c_0$. If we can show that there exists $c\geq c_0$ such that $v_c^+=0$, then $u_0\leq c u_m$ for that $c$ and we are done. Assume now that $v_c^+\neq 0$ for all $c\geq c_0$. Let $(c_n)$ be an increasing sequence in $[c_0,\infty)$ with $c_n\to\infty$ as $n\to\infty$. Then it makes sense to define
  \begin{equation*}
    w_n:=\frac{v_{c_n}^+}{\|v_{c_n}^+\|_E}
  \end{equation*}
  for all $n\in\mathbb N$. Since $u_m(x)>0$ for all $x\in\Omega$ it follows that $w_n(x)\to 0$ for all $n\in\mathbb N$, that is, $w_n\to 0$ pointwise. For the remaining part we have to provide slightly different arguments depending on the space.

  (i) In case $E=L_p(\Omega)$, $p\in(1,\infty)$, this implies that $w_n\rightharpoonup 0$ weakly as $n\to \infty$. Indeed, for every $\varphi\in C_c(\Omega)$ we have
  \begin{equation*}
    \lim_{n\to\infty}\int_\Omega w_n \varphi\,dx = 0
  \end{equation*}
  since $w_n$ becomes zero eventually on every given compact subset of $\Omega$.  As $C_c(\Omega)$ is dense in $L^{p'}(\Omega)$ and $(w_n)$ is bounded it follows that $w_n\rightharpoonup 0$. By the compactness of $A^{-1}$ we deduce that $A^{-1}w_n\to 0$ strongly in $E$. It follows from \eqref{eq:vc-ineq} that $w_n\to 0$ strongly in $E$. This is a contradiction since $\|w_n\|_E=1$ for all $n\in\mathbb N$.

  (ii) If $E=L^1(\Omega)$, then by the compactness of $A^{-1}$ there exists a subsequence $(w_{n_k})$ such that $A^{-1}w_{n_k}$ converges in $L^1(\Omega)$. As $w_{n_k}\to 0$ pointwise, \eqref{eq:vc-ineq} and a version of the Lebesgue dominated convergence theorem (see e.g. \cite[Theorem~3.25]{alt:16:lfa}) imply that $w_{n_k}\to 0$ in $L^1(\Omega)$ as $k\to \infty$. Again this is a contradiction as in~(i).
\end{proof}

\section{The semi-linear logistic equation with degeneracy}
\label{sec:logistic}
The purpose of this section is to demonstrate how the functional analytic tools we developed apply to the abstract logistic equation. Throughout we assume that $\Omega\subseteq\mathbb R^N$ is a bounded domain and that $-A$ is the generator of a positive, irreducible, holomorphic $C_0$-semigroup $(T(t))_{t\geq 0}$ on $E=L^p(\Omega)$ with $1\leq p<\infty$. We also assume that $(T(t))_{t\geq 0}$ is locally smoothing and ultra-contractive in the sense of Definitions~\ref{def:locally-smoothing} and~\ref{def:ultra-contractive}. Thus $T(t)E\subseteq BC(\Omega)$ for all $t>0$ and it follows from Proposition~\ref{prop:bounded-compact} that $T(t)$ is compact for all $t>0$ and hence that $A$ has compact resolvent. Examples of such semigroups are given in Sections~\ref{sec:examples-laplacian} for the Laplace operator with diverse boundary conditions. Additional comments on the semilinear problem for that case are found in Subsection~\ref{sec:semilinear-laplacian}. More general elliptic operators are discussed in Section~\ref{sec:elliptic-divergence-form}.

Consider the logistic equation
\begin{equation}
  \label{eq:logeq}
  Au=\lambda u-m(x)g(x,u)u
\end{equation}
on $E$, where we assume that

\begin{enumerate}[(N1)]
\item $\lambda\in\mathbb R$;
\item $0<m\in L^\infty(\Omega)$;
\item $g\in C^{0,1}\bigl(\bar\Omega\times[0,\infty)\bigr)$ with $g(x,0)=0$ for
  all $x\in\bar\Omega$;\label{g:zero}
\item $\dfrac{\partial g}{\partial \xi}\in C\bigl(\bar\Omega\times [0,\infty)\bigr)$
  and $\dfrac{\partial g}{\partial \xi}(x,\xi)>0$ for all
  $(x,\xi)\in\Omega\times[0,\infty)$\label{g:increasing}
\item $\lim_{\xi\to\infty}g(x,\xi)=\infty$ uniformly with respect to $x$
  in compact subsets of $\Omega$.\label{g:unbounded}
\end{enumerate}

We can assume without loss of generality that $g$ is defined on $\bar\Omega\times\mathbb R$ by taking the odd extension $g(x,\xi):=-g(x,-\xi)$ for all $\xi<0$ and $x\in\bar\Omega$. Then $g\in C(\bar\Omega\times\mathbb R)$ and $\xi\mapsto g(x,\xi)$ is differentiable with respect to $\xi\in\mathbb R$ with a partial derivative that is continuous on $\bar\Omega\times\mathbb R$. This formulation of the logistic equation was considered in \cite{daners:18:gdg}, where very special semigroups on smooth domains were considered.

By a \emph{solution} of \eqref{eq:logeq} we understand a function $u\in D(A)$ such that
\begin{equation*}
  (Au)(x)=\lambda u(x)-m(x)g(x,u(x))u(x)
\end{equation*}
for almost all $x\in \Omega$. For some special situations there is also a notion of \emph{weak solution} to \eqref{eq:logeq}, see Subsection~\ref{sec:semilinear-laplacian}. We will show there that they are solutions in the above sense and hence the theory we develop applies to weak solutions as well.

We are interested in characterising the range of $\lambda\in\mathbb R$ for which \eqref{eq:logeq} has a \emph{non-trivial positive solution} $u$, that is, $0<u\in D(A)$ and $u$ satisfies \eqref{eq:logeq}.  Recall that for $u\in L^p(\Omega)$, $u>0$ means that $u(x)\geq 0$ almost everywhere, and that $u(x)$ does not vanish almost everywhere. Unlike most earlier work on the subject such as \cite{ouyang:92:pss,fraile:eep,du:99:bsc,delpino:94:pss} we do not make any assumptions on the regularity of the vanishing set
\begin{equation*}
  \Omega_0:=\{x\in\Omega\colon m(x)=0\},
\end{equation*}
nor on the self-adjointness of the elliptic operator. Such regularity
conditions were removed in \cite{daners:18:gdg}, but $\partial\Omega$
was still required to be of class $C^2$. We also remove this condition
and do not make any explicit assumptions on the regularity of
$\partial\Omega$.

As $m>0$ we know from Proposition~\ref{prop:ev-monotone} that $\lambda_1(\gamma m)$ is strictly increasing as a function of $\gamma\geq 0$, where as before $\lambda_1(m):=\lambda_1(A+m)$. Hence
\begin{equation}
  \label{eq:lstar}
  \lambda^*(m):=\lim_{\gamma\to\infty}\lambda_1(\gamma m)\in(-\infty,\infty]
\end{equation}
exists. Our main result is as follows.

\begin{theorem}
  \label{thm:main}
  Under the above assumptions, the logistic equation \eqref{eq:logeq} has a non-trivial positive solution if and only if $\lambda \in\bigl(\lambda_1(0),\lambda^*(m)\bigr)$. In that case the non-trivial positive solution $u_\lambda$ is unique and linearly stable. Moreover, $u_\lambda \in BC(\Omega)$, $u(x)>0$ for all $x\in\Omega$ and
  \begin{equation*}[\lambda\mapsto u_\lambda]
    \in C^1\bigl(\bigl(\lambda_1(0),\lambda^*(m)\bigr),D(A)\bigr)
  \end{equation*}
  is strictly increasing in the sense that $\mu<\lambda$ implies that $u_\mu(x)<u_\lambda(x)$ for all $x\in\Omega$. Finally, $u_\lambda\downarrow 0$ in $L^\infty(\Omega)\cap E$ as $\lambda\downarrow \lambda_1(0)$ and $\|u_\lambda\|_\infty\uparrow\infty$ as $\lambda\uparrow \lambda^*(m)$.
\end{theorem}

While there is no solution of \eqref{eq:logeq} for
$\lambda>\lambda^*(m)$ there are solutions on $\Omega\setminus\Omega_0$
blowing up on $\partial\Omega_0$, see for instance
\cite{cirstea:02:eub,du:06:ost,lopez:16:mpe}.

We will prove the above theorem in a sequence of results. Before we do so we first make some observations.

\begin{remark}
  \label{rem:fixed-point-equation}
  Recall from Section~\ref{sec:strong-monotonicity} that $\omega I + A$ is invertible as an operator on $L^p(\Omega)$ for each $\omega>-\lambda_1(0)$. In that case we also know that $(\omega I+A)^{-1}\geq 0$. For $u\in L^p(\Omega)$ the function $mg(\cdot\,,u)$ is measurable on $\Omega$. Hence, $u\in D(A)$ is a solution of \eqref{eq:logeq} if and only if $mg(\cdot\,,u)u\in L^p(\Omega)$ and
  \begin{equation}
    \label{eq:fixed-pt-eq}
    u=F(u):=(\omega I+A)^{-1}\bigl(\lambda u+\omega u-mg(\cdot\,,u)u\bigr)
  \end{equation}
  for some $\omega>-\lambda_1(0)$. This observation transforms the original problem \eqref{eq:logeq} into a fixed point equation.
\end{remark}

We first show some regularity of positive solutions of \eqref{eq:logeq}, keeping the assumption from the start of the section.

\begin{proposition}[regularity of solutions]
  \label{prop:solution-bounded}
  Let $\lambda\in\mathbb R$ and let $u>0$ be a solution of \eqref{eq:logeq}. Then, $u\in BC(\Omega)$ and $u(x)>0$ for all $x\in\Omega$. Moreover,
  \begin{equation}
    \label{eq:solution-bound}
    0<u\leq e^{\lambda t}T(t)u
  \end{equation}
  for all $t>0$.
\end{proposition}
\begin{proof}
  Let $0<u$ be a solution of \eqref{eq:logeq} and let $\omega > -\lambda_1(0)$. Then by \eqref{eq:fixed-pt-eq} we have
  \begin{equation}
    \label{fixed-pt-ineq}
    0<u\leq(\lambda+\omega)(\omega I+A)^{-1}u.
  \end{equation}
  Iterating this inequality we obtain
  \begin{equation*}
    0<u\leq(\omega+\lambda)^n(\omega I+A)^{-n}u
  \end{equation*}
  for all $n\in\mathbb N$. Let $t>0$ and set $\omega:=n/t$. Then, for $n$ large enough, it follows that
  \begin{equation*}
    0<u\leq\Bigl(1+\frac{\lambda t}{n}\Bigr)^n
    \Bigl[\frac{n}{t}\Bigl(\frac{t}{n} I+A\Bigr)^{-1}\Bigr]^nu.
  \end{equation*}
  Letting $n\to\infty$ we find \eqref{eq:solution-bound} for all $t>0$; see \cite[Corollary~III.5.5]{engel:00:ops}. By assumption $T(t)$ maps to $BC(\Omega)$ and thus $u\in L^\infty(\Omega)$. Hence $\tilde m:=mg(\cdot\,,u)\in L^\infty(\Omega)$ and $Au+\tilde mu=\lambda u$.  Theorem~\ref{thm:m-locally-smoothing} and Theorem~\ref{thm:positivity-improving} imply that $u\in BC(\Omega)$ with $u(x)>0$ for all $x\in\Omega$.
\end{proof}

We next establish a necessary condition for the existence of a positive solution.

\begin{proposition}[Necessary conditions for existence]
  \label{prop:exist-1}
  Suppose that $0<u\in D(A)$ is a positive solution of \eqref{eq:logeq} for some $\lambda\in\mathbb R$. Then $\lambda\in\bigl(\lambda_1(0),\lambda^*(m)\bigr)$ and
  \begin{equation}
    \label{eq:pev-solution}
    \lambda=\lambda_1\bigl(mg(\cdot\,,u)\bigr).
  \end{equation}
\end{proposition}
\begin{proof}
  If $0<u\in D(A)$ is a solution of \eqref{eq:logeq}, then $u\in BC(\Omega)$ by Proposition~\ref{prop:solution-bounded}. Hence, looking at \eqref{eq:logeq} we see that $u>0$ is an eigenvector for the eigenvalue problem $Au+\tilde mu=\lambda u$ with $\tilde m:=mg(\cdot\,,u)\in L^\infty(\Omega)$. Hence $\lambda=\lambda_1(\tilde m)$ by the uniqueness of the principal eigenvalue from Proposition~\ref{prop:A-pev}.  This proves \eqref{eq:pev-solution}. As $u(x)>0$ for all $x\in\Omega$ and $m>0$, assumption (N\ref{g:increasing}) on $g$ implies that
  \begin{equation*}
    0<\tilde m=mg(\cdot\,,u)\leq \gamma m<\infty,
  \end{equation*}
  where $\gamma:=\|g(\cdot\,,u)\|_\infty$. Now $\lambda_1(0)<\lambda<\lambda^*(m)$ by Proposition~\ref{prop:ev-monotone} and \eqref{eq:lstar}.
\end{proof}

We next prove the uniqueness of the non-trivial solution. The argument is different from the commonly used one such as that in \cite[Theorem~5.1]{du:06:ost} or \cite[Theorem~1.7]{lopez:16:mpe}. We completely avoid the use of sub- and super-solutions and the Hopf boundary maximum principle and only work with spectral properties.

\begin{proposition}[Uniqueness of positive solutions]
  \label{prop:uniqueness}
  For every $\lambda\in\mathbb R$ the problem \eqref{eq:logeq} has at most one positive solution $0<u\in D(A)$.
\end{proposition}
\begin{proof}
  Fix $\lambda\in\mathbb R$ and let $0<u,v\in D(A)$ be solutions of \eqref{eq:logeq}. Set $w:=u-v$. Then
  \begin{equation}
    \label{eq:logeq-unique}
    Aw=\lambda w-m\bigl(g(\cdot\,,u)u-g(\cdot\,,v)v\bigr)
    =\lambda w-mg(\cdot\,,u)w - mVw,
  \end{equation}
  where we have rewritten the nonlinear terms in the form
  \begin{equation*}
    g(\cdot\,,u)u-g(\cdot\,,v)v
    =g(\cdot\,,u)w+\bigl(g(\cdot\,,u)-g(\cdot\,,v)\bigr)v
    =g(\cdot\,,u)w+Vw
  \end{equation*}
  with
  \begin{equation*}
    V:=\int_0^1\frac{\partial g}{\partial\xi}(\cdot\,,v+sw)\,ds\,v.
  \end{equation*}
  By assumption (N\ref{g:increasing}) on $g$ and since $v(x)>0$ for all $x\in\Omega$ it follows from Proposition~\ref{prop:solution-bounded} that $V(x)>0$ for all $x\in\Omega)$. As $m>0$ we therefore have that $mV>0$ and thus by Proposition~\ref{prop:ev-monotone} and \eqref{eq:pev-solution}
  \begin{equation}
    \label{eq:ev-mV}
    \lambda_1\bigl(m(g(\cdot\,,u)+mV)\bigr)
    >\lambda_1\bigl(mg(\cdot\,,u)\bigr)
    =\lambda.
  \end{equation}
  If $w\neq 0$, then \eqref{eq:logeq-unique} implies that $\lambda\in\sigma\bigl(A+mg(\cdot\,,u)+mV\bigr)$. By Proposition~\ref{prop:A-pev} and \eqref{eq:ev-mV}
  \begin{equation*}
    \lambda
    \geq\lambda_1\bigl(mg(\cdot\,,u)+mV\bigr)
    >\lambda,
  \end{equation*}
  which is clearly not possible. Hence we must have $w=0$, that is, $u=v$.
\end{proof}

Our next aim is to prove the existence of a non-trivial positive solution of \eqref{eq:logeq}. We know from Proposition~\ref{prop:solution-bounded} that any such solution is bounded. Hence we will seek such a solution in a bounded subset of positive functions $B\subseteq L^\infty(\Omega)\subseteq L^p(\Omega)$. Given such a bounded set, by choosing $\omega>-\lambda_1(0)$ large enough, we can guarantee that the mapping $F\colon B\to L^p(\Omega)$ given by \eqref{eq:fixed-pt-eq} is increasing. Indeed, if $k:=\sup_{u\in B}\|u\|_\infty$, then we choose $\omega>-\lambda_1(0)$ such that such that
\begin{equation}
  \label{eq:restriction-positive}
  \lambda+\omega-m(x)g(x,\xi)-m(x)\frac{\partial g}{\partial\xi}(x,\xi)\xi\geq 0
\end{equation}
for all $x\in\bar\Omega$ and all $\xi\in[0,k]$.

In order to prove the existence of a solution we use the Monotone Fixed Point Theorem due to Tarski \cite[Theorem~1]{tarski:55:ltf}. For completeness we include a proof. We point out that this fixed point theorem does not rely on the continuity of the map.

\begin{theorem}[Tarski]
  \label{thm:tarski}
  Let $S$ be a partially ordered set and let $a,b\in S$ with $a\leq b$. Let $[a,b]:=\{u\in S\colon a\leq u\leq b\}$. Assume that every non-empty subset of $[a,b]$ has a supremum in $[a,b]$. Let $G\colon[a,b]\to S$ be increasing, that is, if $u\leq v$, then $G(u)\leq G(v)$. Assume that $a\leq G(a)$ and that $b\geq G(b)$. Then $G$ has a fixed point in $[a,b]$.
\end{theorem}
\begin{proof}
  Let $M:=\{u\in [a,b]\colon u\leq G(u)\}$. Then $a\in M$, so $M$ is non-empty. Hence $u^*:=\sup(M)\in[a,b]$ exists. Since $G$ is increasing on $[a,b]$ we have
  \begin{equation*}
    u\leq G(u)\leq G(u^*)\leq G(b)\leq b
  \end{equation*}
  for all $u\in M$. By taking a supremum on the left hand side $u^*\leq G(u^*)=:v$. As $G$ is increasing we therefore have $v=G(u^*)\leq G(v)$. Thus $v\in M$ and since $u^*=\sup(M)$ we have $v\leq u^*$. Hence $v=u^*$ and consequently $G(u^*)=u^*$.
\end{proof}

In the above theorem $a$ is called a \emph{sub-solution} and $b$ is called a \emph{super-solution} of $u=G(u)$.

\begin{remark}
  From the proof of the above theorem, the fixed point $u^*$ by construction is the largest possible fixed point of $G$ in $[a,b]$. If we assume that every subset of $[a,b]$ has an infimum in $[a,b]$, then by similar arguments $u_*:=\inf\{u\in[a,b]\colon u\geq G(u)\}$ is the smallest possible fixed point of $G$ in $[a,b]$. Moreover, $u_*\leq u^*$.
\end{remark}

Next we consider the mapping $F\colon L^p(\Omega)\to L^p(\Omega)$ given by \eqref{eq:fixed-pt-eq}. Recall from Remark~\ref{rem:fixed-point-equation} that $u$ is a solution of \eqref{eq:logeq} if and only if $u$ is a fixed point of $F$. In order to apply Tarski's Theorem to get the existence of such a fixed point we need to find a subsolution and a supersolution of $F$. Note that $0<u\in D(A)$ is a supersolution for $F$ if $Au\geq \lambda u-mg(\cdot\,,u)u$ and a subsolution if $Au\leq \lambda u-mg(\cdot\,,u)u$. The idea for the construction of a supersolution is similar to that in \cite[Proposition~3.2]{daners:18:gdg}, which in turn generalised constructions given in \cite{du:03:bbs,du:06:dlm}. Our approach works in an abstract context and does not require much regularity.

\begin{proposition}[Existence of supersolution]
  \label{prop:supersol}
  Let $\lambda<\lambda^*(m)$. For $\delta>0$ define
  $\Omega_\delta :=\{x\in\Omega\colon\dist(x,\partial\Omega)>\delta\}$
  and let $m_\delta:=1_{\Omega_\delta }m$. Then the following assertions are true.
  \begin{enumerate}[\upshape (i)]
  \item There exist $\delta>0$ and $\gamma>0$ such that $\lambda<\lambda_1(\gamma m_\delta)<\lambda^*(m)$.
  \item Let $\varphi$ be the principal eigenvector of $A+\gamma m1_{\Omega_\delta}$ corresponding to the principal eigenvalue $\lambda_1(\gamma m_\delta)$ with $\delta,\gamma>0$ as in   {\upshape (i)}. Then there exists $\kappa_0>0$ such that $\kappa\varphi$ is a positive supersolution of \eqref{eq:logeq} whenever $\kappa\geq\kappa_0$.
  \end{enumerate}
\end{proposition}
In the statement of the theorem $1_{\Omega_\delta}$ is the indicator function of $\Omega_\delta$ which takes the value one on $\Omega_\delta$ and zero otherwise.
\begin{proof}
  (i) Let $\lambda<\lambda^*(m)$. By definition of $\lambda^*(m)$ there exists $\gamma>0$ such that $\lambda<\lambda_1(\gamma m)<\lambda^*(m)$. Clearly $m_\delta\uparrow m$ pointwise on $\Omega$ and hence by continuity of the principal eigenvalue with respect to the weight (see Theorem~\ref{thm:ev-continuous}) there exists $\delta>0$ such that $\lambda<\lambda_1(\gamma m_\delta)\leq\lambda_1(\gamma   m)<\lambda^*(m)$.

  (ii) Let $0<\varphi$ be the principal eigenvector corresponding to $\lambda_1(\gamma m_\delta)$. By (i)
  \begin{equation*}
    \begin{split}
      A(\kappa\varphi) &=\lambda_1(\gamma m_\delta)\kappa\varphi-\gamma m_\delta\kappa\varphi\\ &=\lambda\kappa\varphi-mg(\cdot\,,\kappa\varphi)\kappa\varphi +\bigl(\lambda_1(\gamma m_\delta)-\lambda\bigr)\kappa\varphi +\bigl(mg(\cdot\,,\kappa\varphi)-\gamma m_\delta\bigr) \kappa\varphi\\ &>\lambda\kappa\varphi-mg(\cdot\,,\kappa\varphi)\kappa\varphi +\bigl(mg(\cdot\,,\kappa\varphi)-\gamma m_\delta\bigr) \kappa\varphi\\
    \end{split}
  \end{equation*}
  for all $\kappa>0$. In order for $\kappa\varphi$ to be a supersolution it remains to show that
  \begin{equation}
    \label{eq:kappa-phi}
    mg(\cdot\,,\kappa\varphi)-\gamma m_\delta\geq 0
  \end{equation}
  for all $\kappa$ sufficiently large. By construction, $m_\delta=0$ on $\Omega\setminus\Omega_\delta$ and hence
  \begin{equation*}
    mg(\kappa\varphi)-\gamma m_\delta
    =mg(\kappa\varphi)\geq 0
    \qquad\text{on $\Omega\setminus\Omega_\delta$.}
  \end{equation*}
  As $\bar\Omega_\delta\subseteq\Omega$ is compact and $\varphi\in BC(\Omega)$ it follows from Theorem~\ref{thm:positivity-improving} that there exists $c>0$ such that $\varphi(x)\geq c$ for all $x\in\bar\Omega_\delta$. Hence, by assumption (N\ref{g:unbounded}) on $g$
  \begin{equation*}
    \lim_{\kappa\to\infty}\bigl(\inf_{x\in\Omega_\delta}g(x,\kappa\varphi)\bigr)
    \geq\lim_{\kappa\to\infty}\bigl(\inf_{x\in\Omega_\delta}g(x,\kappa c)\bigr)
    =\infty.
  \end{equation*}
  We can therefore choose $\kappa_0>0$ such that \eqref{eq:kappa-phi} is valid on $\Omega_\delta$ for all $\kappa\geq\kappa_0$. Hence \eqref{eq:kappa-phi} is valid on $\Omega$ and thus $\kappa\varphi$ is a supersolution of \eqref{eq:logeq} for all $\kappa\geq\kappa_0$.
\end{proof}

Given $\lambda>\lambda_1(0)$ we find a subsolution of \eqref{eq:logeq} in the usual way. A subsolution is a function $u\in D(A)$ with $Au\leq \lambda u+mg(\cdot,u)u$. We include the short proof.

\begin{lemma}[existence of subsolution]
  \label{lem:subsolution}
  Let $\lambda>\lambda_1(0)$ and let $0<\psi$ be the principal eigenvector of $A$ corresponding to $\lambda_1(0)$. Then there exists $\varepsilon_0>0$ such that $\varepsilon\psi$ is a subsolution of \eqref{eq:logeq} for all $\varepsilon\in(0,\varepsilon_0]$.
\end{lemma}
\begin{proof}
  Given $\varepsilon>0$ we have
  \begin{equation*}
    A(\varepsilon\psi)
    =\lambda_1(0)\varepsilon\psi
    =\lambda\varepsilon\psi-mg(\varepsilon\psi)\varepsilon\psi
    -\bigl(\lambda-\lambda_1(0)-mg(\varepsilon\psi)\bigr)\varepsilon\psi.
  \end{equation*}
  For $\varepsilon\psi$ to be a subsolution we need that $\lambda-\lambda_1(0)-mg(\varepsilon\psi)\geq 0$. By assumption $g(x,\cdot)$ is strictly increasing, $g(\cdot,0)=0$ and $\lambda-\lambda_1(0)>0$. As $\psi\in BC(\Omega)$ and $g\in C(\bar\Omega\times[0,\infty))$ there exits $\varepsilon_0>0$ such that
  \begin{equation*}
    0\leq mg\bigl(x,\varepsilon\psi(x)\bigr)
    \leq mg\bigl(x,\varepsilon\|\psi\|_\infty\bigr)
    <\lambda-\lambda_1(0)
  \end{equation*}
  for all $\varepsilon\in (0,\varepsilon_0]$ and $x\in\Omega$ as required.
\end{proof}

The following existence theorem makes critical use one of the key new features from this paper, namely the eigenvector comparison Theorem~\ref{thm:ev-comparison}. In the absence of boundary regularity it guarantees that the sub- and supersolutions we constructed can be ordered.

\begin{proposition}[Existence of a positive solution]
  \label{prop:exist-2}
  For every $\lambda\in\bigl(\lambda_1(0),\lambda^*(m)\bigr)$ the problem \eqref{eq:logeq} has a positive solution.
\end{proposition}
\begin{proof}
  Let $\lambda\in\bigl(\lambda_1(0),\lambda^*(m)\bigr)$. Let $\psi$ be the principal eigenvector of $A$. By Lemma~\ref{lem:subsolution} there exists $\varepsilon>0$ such that $\varepsilon\psi$ is a subsolution of \eqref{eq:logeq}. Let $\varphi$ be the principal eigenvector of $A+\gamma m_\delta$ with $\gamma$ and $\delta>0$ as in Proposition~\ref{prop:supersol}. That proposition implies the existence of $\kappa_0$ such that $\kappa\psi$ is a supersolution of \eqref{eq:logeq} for all $\kappa\geq\kappa_0$. As $m_\delta$ has compact support in $\Omega$, Theorem~\ref{thm:ev-comparison} implies the existence of $\kappa>\kappa_0$ so that $\varepsilon\varphi<\kappa\psi$. Hence we have an ordered pair of sub- and supersolutions of \eqref{eq:logeq}. Solutions of \eqref{eq:logeq} correspond to fixed points of \eqref{eq:fixed-pt-eq}. As discussed before Tarski's Theorem the nonlinearity $F$ in \eqref{eq:fixed-pt-eq} is monotone on the order interval $[\varepsilon\varphi,\kappa\psi]$ if we choose $\omega>-\lambda_1(0)$ such that \eqref{eq:restriction-positive} holds for all $x\in\bar\Omega$ and all $\xi\in[0,\kappa\|\psi\|_\infty]$.  Now the existence of a solution follows from Tarski's fixed point Theorem~\ref{thm:tarski} since every non-empty set in $[\varepsilon\psi,\kappa\varphi]$ has a supremum in $L^p(\Omega)$, and that supremum lies in $[\varepsilon\psi,\kappa\varphi]$.
\end{proof}

\begin{remark}
  \label{rem:sub-solution-comparison}
  In the proof of Proposition~\ref{prop:exist-2} we have used a specific pair of ordered sub- and super-solutions to prove the existence of a solution $u_\lambda>0$ of \eqref{eq:logeq}. That specific pair can be replaced by any ordered pair of sub- and super-solutions $0<\underline u,\overline u\in E\cap L^\infty(\Omega)$. It will lead to the existence of a solution $w_\lambda\in[\underline u,\overline u]$ of \eqref{eq:logeq}. Proposition~\ref{prop:uniqueness} asserts that the positive solution is unique. In particular $w_\lambda = u_\lambda$ and hence $\underline u\leq u_\lambda\leq \overline u$.
\end{remark}

We finally prove that any non-trivial positive solution of \eqref{eq:logeq} is a linearly stable stationary solution of the parabolic equation
\begin{equation}
  \label{eq:logeq-parabolic}
  \begin{aligned}
    \frac{du}{dt}+Au & =\lambda u-g(\cdot\,,u)u &  & \text{for $t\geq 0$,} \\
    u(0)             & =u_0.                    &  &                       \\
  \end{aligned}
\end{equation}

\begin{proposition}[Stability of positive solutions]
  \label{prop:stability}
  Every non-trivial positive solution of \eqref{eq:logeq} is a linearly stable equilibrium solution of \eqref{eq:logeq-parabolic}.
\end{proposition}
\begin{proof}
  Let $u$ be a positive solution of \eqref{eq:logeq}. The linearization of \eqref{eq:logeq} about $u$ is
  \begin{equation*}
    Av+mg(\cdot\,,u)v-\lambda v+m\frac{\partial g}{\partial \xi}(\cdot\,,u)uv
    =0.
  \end{equation*}
  By \eqref{eq:pev-solution} we know that $\lambda_1\bigl(mg(\cdot\,,u)-\lambda\bigr)=0$. By assumption (N\ref{g:increasing}) and since $u(x)>0$ for all $x\in\Omega$ we have that $m\frac{\partial g}{\partial \xi}(\cdot\,,u)u>0$. Hence Proposition~\ref{prop:ev-monotone} implies that
  \begin{equation*}
    \lambda_1\Bigl(mg(\cdot\,,u)
    -\lambda+m\frac{\partial g}{\partial \xi}(\cdot\,,u)u\Bigr)
    >\lambda_1\bigl(mg(\cdot\,,u)-\lambda\bigr)
    =0,
  \end{equation*}
  which means that $u$ is linearly stable.
\end{proof}

\begin{proposition}[Monotonicity and differentiability]
  \label{prop:solution-increasing}
  Let $\lambda\in(\lambda_1(0),\lambda^*(m))$ and $u_\lambda>0$ the unique positive solution of \eqref{eq:logeq}. Then $[\lambda\to u_\lambda]\in   C^1\bigl((\lambda_1(0),\lambda^*(m)),E\bigr)$ is pointwise strictly increasing on $[\lambda_1(0),\lambda^*(m))$. Moreover, $u_\lambda\downarrow 0$ in $L^\infty(\Omega)\cap E$ as $\lambda\downarrow \lambda_1(0)$ and $\|u_\lambda\|_\infty\uparrow\infty$ as $\lambda\uparrow\lambda^*(m)$.
\end{proposition}
\begin{proof}
  Suppose that $\lambda_1(0)<\lambda<\mu<\lambda^*(m)$. Set $w:=u_\mu-u_\lambda$. Subtracting the equations we see that
  \begin{equation*}
    \begin{split}
      Aw&=\mu u_\mu-\lambda u_\lambda -mg(\cdot\,,u_\mu)u_\mu+mg(\cdot\,,u_\lambda)u_\lambda\\ &=(\mu-\lambda)u_\lambda+\bigl(\mu-mg(\cdot\,,u_\mu)\bigr)(u_\mu-u_\lambda) -m\bigl(g(\cdot\,,u_\mu)-g(\cdot\,,u_\lambda)\bigr)u_\lambda\\ &=(\mu-\lambda)u_\lambda -\bigl(mg(\cdot\,,u_\mu)-\mu+mV_{\mu,\lambda}\bigr)w\\
    \end{split}
  \end{equation*}
  if we set
  \begin{equation*}
    V_{\mu,\lambda}
    :=\int_0^1\frac{\partial g}{\partial\xi}(\cdot\,,u_\lambda+sw)\,ds\,u_\lambda.
  \end{equation*}
  We deduce that
  \begin{equation}
    \label{eq:solutions-diff}
    \bigl(A+mg(\cdot\,,u_\mu)-\mu+mV_{\mu,\lambda}\bigr)(u_\mu-u_\lambda)
    =(\mu-\lambda)u_\lambda
  \end{equation}
  Arguing as in the proof of Proposition~\ref{prop:uniqueness} we see that $mV_{\mu,\lambda}>0$. Hence \eqref{eq:pev-solution} with $\lambda=\mu$ and Proposition~\ref{prop:ev-monotone} imply that $\lambda_1(mg(\cdot\,,u_\mu)-\mu+mV_{\mu,\lambda})>0$. As $u_\lambda>0$ it follows from \eqref{eq:solutions-diff} that $u_\mu(x)-u_\lambda(x)>0$ for all $x\in\Omega$ whenever $\mu-\lambda>0$. Hence $\lambda\mapsto u_\lambda(x)$ is strictly increasing for all $x\in\Omega$.

  We next prove the continuity of $\lambda\mapsto u_\lambda$. By the monotonicity $u_\mu\to v$ pointwise as $\mu\uparrow\lambda$ and hence in $E$ since $E$ has order continuous norm (or by the dominated convergence theorem). Moreover, $v\leq\|u_\lambda\|_\infty$. Taking $\omega>-\lambda_1(0)$ it follows from Remark~\ref{rem:fixed-point-equation} that $u_\mu=F(u_\mu)$.  Letting $\mu\to\lambda$ we deduce that $v=F(v)$ and thus $v=u_\lambda$.  As $u_\mu\leq u_{\lambda+\delta}$ if $\lambda<\mu<\lambda+\delta<\lambda^*(m)$, a similar argument applies to the right limit as $\mu\downarrow\lambda$, proving the continuity of $\lambda\mapsto u_\lambda$. To get convergence in $D(A)$ note that $mg(\cdot\,,u_\mu)$ is uniformly bounded for $\mu$ in a neighbourhood of $\lambda$ and that $mg(\cdot\,,u_\mu)\to mg(\cdot\,,u_\lambda)$ pointwise. Hence
  \begin{equation*}
    \lim_{\mu\to\lambda}Au_\mu
    =\lim_{\mu\to\lambda}\bigl(\mu-mg(\cdot\,,u_\mu)\bigr)u_\mu
    =\bigl(u_\lambda-mg(\cdot\,,u_\lambda)\bigr)u_\lambda
    =Au_\lambda
  \end{equation*}
  in $E$ and thus $u_\mu\to u_\lambda$ in $D(A)$ as $\mu\to\lambda$.

  We next investigate what happens in the limit cases. First consider the limit as $\lambda\downarrow\lambda_1(0)$. As in the proof of continuity $u_*(x):=\lim_{\lambda\downarrow\lambda_1(0)}u_\lambda(x)$ exists for all $x\in\Omega$. By the same argument as used to prove the continuity, $u_*\geq 0$ satisfies the equation $Au_*=\lambda_1(0)u_*-mg(\cdot\,, u_*)u_*$. Now it follows from Proposition~\ref{prop:exist-1} that $u_*=0$. It follows from Proposition~\ref{prop:solution-bounded} that
  \begin{equation*}
    0<u_\lambda\leq e^{t(\lambda_1(m)+1)}T(t)u_\lambda
  \end{equation*}
  for every $t>0$ and $\lambda_1(0)<\lambda<\min\{\lambda_1(0)+1,\lambda^*(m)\}$. By assumption $T(t)E\subseteq L^\infty(\Omega)$ and hence by the closed graph theorem $T(t)\in\mathcal L\bigl(E,L^\infty(\Omega)\bigr)$. Hence there exists a constant $C>0$ such that $\|u_\lambda\|_\infty\leq C\|u_\lambda\|_E$ and thus $\|u_\lambda\|_\infty\to 0$ as $\lambda\downarrow 0$.

  We proceed similarly if $\lambda\uparrow\lambda^*(m)$. We give a proof by contradiction assuming that $\lim_{\lambda\uparrow\lambda^*(m)}\|u_\lambda\|_\infty=M<\infty$. In that case $u^*(x):=\lim_{\lambda\uparrow\lambda_1(0)}u_\lambda(x)$ exists for all $x\in\Omega$. Again, the same argument as in case of continuity applies and $0<u^*\in E$ satisfies \eqref{eq:logeq} with $\lambda=\lambda^*(m)$. However, this is impossible by Proposition~\ref{prop:exist-1}.

  To prove the differentiability of $\lambda\mapsto u_\lambda$ note that due to the continuity
  \begin{equation}
    \label{eq:cont-potential}
    \lim_{\mu\to\lambda}\bigl(mg(\cdot\,,u_\mu)-\mu+mV_{\mu,\lambda}\bigr)
    =mg(\cdot\,,u_\lambda)-\lambda+m\frac{\partial
      g}{\partial\xi}(\cdot\,,u_\lambda)u_\lambda
  \end{equation}
  in $E$. By assumption (N\ref{g:increasing}) on $g$ and since $u(x)>0$ for all $x\in\Omega$ we have that
  \begin{equation*}
    m\frac{\partial g}{\partial\xi}(\cdot\,,u_\lambda)u_\lambda>0.
  \end{equation*}
  Hence, by \eqref{eq:pev-solution} and Proposition~\ref{prop:ev-monotone}
  \begin{equation*}
    \lambda_1\Bigl(mg(\cdot\,,u_\lambda)-\lambda+m\frac{\partial
      g}{\partial\xi}(\cdot\,,u_\lambda)u_\lambda\Bigr)>0
  \end{equation*}
  It follows from Theorem~\ref{thm:ev-continuous} that
  \begin{equation*}
    \lim_{\mu\to\lambda}\bigl(A+mg(\cdot\,,u_\mu)-\mu+mV_{\mu,\lambda}\bigr)^{-1}
    =\Bigl(A+mg(\cdot\,,u_\lambda)-\lambda+m\frac{\partial
      g}{\partial\xi}(\cdot\,,u_\lambda)u_\lambda\Bigr)^{-1}
  \end{equation*}
  in $\mathcal L(E)$. Rearranging \eqref{eq:solutions-diff} we see that for any $\lambda,\mu\in\bigl(\lambda_1(0),\lambda^*(m)\bigr)$
  \begin{equation*}
    u_\mu
    =u_\lambda
    +\left[\bigl(A+mg(\cdot\,,u_\mu)-\mu+mV_{\mu,\lambda}\bigr)^{-1}u_\lambda\right]
    (\mu-\lambda).
  \end{equation*}
  As $u_\lambda\in D(A)$ we deduce that
  \begin{equation*}
    \begin{split}
      \lim_{\mu\to\lambda}\bigl(A+mg(\cdot\,,u_\mu)
      &-\mu+mV_{\mu,\lambda}\bigr)^{-1}u_\lambda\\
      &=\Bigl(A+mg(\cdot\,,u_\lambda)-\lambda+m\frac{\partial
        g}{\partial\xi}(\cdot\,,u_\lambda)u_\lambda\Bigr)^{-1}u_\lambda\\
    \end{split}
  \end{equation*}
  in $D(A)$. Hence $\lambda\mapsto u_\lambda$ is differentiable with
  \begin{equation*}
    v_\lambda:=\frac{du_\lambda}{d\lambda}
    =\Bigl(A+mg(\cdot\,,u_\lambda)-\lambda+m\frac{\partial
      g}{\partial\xi}(\cdot\,,u_\lambda)u_\lambda\Bigr)^{-1}u_\lambda;
  \end{equation*}
  see \cite[Section~2]{arora:20:aaf}. In particular the derivative is the unique solution of the equation
  \begin{equation}
    \label{eq:log-eq-derivative}
    Av_\lambda+mg(\cdot\,,u_\lambda)v_\lambda-\lambda v_\lambda
    +m\frac{\partial g}{\partial\xi}(\cdot\,,u_\lambda)u_\lambda v_\lambda
    =u_\lambda
  \end{equation}
  in $E$ as expected.
\end{proof}

As a final remark we note that if we are prepared to work with solutions that are \emph{a priori} in $L^\infty(\Omega)$ rather than proving it in Proposition~\ref{prop:solution-bounded}, then we do not really need the ultra-contractivity of the semigroup generated by $-A$. The eigenvector comparison theorem in Section~\ref{sec:eigenvector-comparison} does not rely on that, it only relies on the local smoothing property. However, we do need to know that the eigenvectors used to construct to sub- and super-solutions in Proposition~\ref{prop:supersol} and Lemma~\ref{lem:subsolution} are in $L^\infty(\Omega)$. Then the arguments in the proof of Proposition~\ref{prop:exist-2} still work.  We did not set up the theory that way since, apart from the ultra-contractivity of the semigroup, we are not aware of any other criteria that guarantee that the principal eigenvectors are in $L^\infty(\Omega)$.

\section{The Laplacian with diverse boundary conditions}
\label{sec:examples-laplacian}
In this section we consider an open, bounded connected set $\Omega\subseteq\mathbb R^N$ and realisations of the Laplacian with Dirichlet-, Robin- and Neumann boundary conditions. We will show that the operator generates a $C_0$-semigroup which satisfies the requirements of Theorem~\ref{thm:main}: positivity, irreducibility, compactness, smoothing and ultra-contractivity. Only for the Neumann Laplacian we need some weak regularity conditions on the boundary of $\Omega$ to guarantee the ultra-contractivity.

\subsection{The Laplacian with Dirichlet boundary conditions}
\label{sec:dirichlet-laplacian}
Let $\Omega$ be open, bounded and connected. The \emph{Laplacian with Dirichlet boundary conditions} or briefly the \emph{Dirichlet Laplacian} is the operator $\Delta_D$ on $L^2(\Omega)$ defined by
\begin{equation}
  \label{eq:dirichlet-laplacian}
  \begin{aligned}
    D(\Delta_D) & =\{u\in H_0^1(\Omega)\colon\Delta u\in L^2(\Omega)\}, \\
    \Delta_Du   & :=\Delta u\qquad\text{for $u\in D(\Delta_D)$.}
  \end{aligned}
\end{equation}
It is a self-adjoint dissipative operator and so $\Delta_D$ generates a contractive $C_0$-semigroup $(T_D(t))_{t\geq 0}$ on $L^2(\Omega)$. This semigroup fulfils the hypotheses of Theorem~\ref{thm:main}. In fact, the following assertions are true.

\begin{theorem}[Dirichlet Laplacian]
  \label{thm:dirichlet-laplacian}
  The $C_0$-semigroup $(T_D(t))_{t\geq 0}$ is holomorphic, positive, irreducible, locally smoothing and ultra-contractive.
\end{theorem}
\begin{proof}
  Positivity follows from the Beurling-Deny criterion; see \cite[Section~1.3]{davies:89:hks} or \cite[Theorem~2.7]{ouhabaz:05:ahe}. An elegant criterion due to Ouhabaz implies irreducibility; see \cite[Theorem~2.9]{ouhabaz:05:ahe}. The semigroup is dominated by the Gaussian semigroup on $\mathbb R^N$, that is,
  \begin{equation*}
    |T_D(t)u|\leq G(t)|\tilde u|
  \end{equation*}
  for all $t>0$, where $\tilde u$ is the extension of $u$ to $\mathbb R^N$ by zero, see for instance \cite[Example~7.4.1(a)]{arendt:04:see}. The Gaussian semigroup given by
  \begin{equation*}[G(t)u](x):=(4\pi t)^{-N/2}\int_{\mathbb R^N}e^{-|y-x|^2/4t}u(y)\,dy
  \end{equation*}
  for all $x\in\mathbb R^N$.  Thus $T_D(t)L^2(\Omega)\subseteq L^\infty(\Omega)$ for all $t>0$. As the semigroup is holomorphic one has $T(t)L^2(\Omega)\subseteq D(\Delta_D^k)$ for all $k\in\mathbb N$ and all $t>0$. It follows from elliptic regularity that $D(\Delta_D^k)\subseteq H_{\loc}^{2k}(\Omega)$ for all $k\in\mathbb N$, see for instance \cite[Theorem~6.59]{arendt:22:pde}. By standard Sobolev embedding theorems we have $H_{\loc}^{2k}(\Omega)\subseteq C(\Omega)$ if $k>N/4$; see for instance \cite[Theorem~6.58]{arendt:22:pde}. This shows that $T(t)L^2(\Omega)\subseteq C(\Omega)$ for all $t>0$. In order to show that $(T_D(t))_{t\geq 0}$ is locally smoothing it remains to show that for each $x\in\Omega$ there exists $t>0$ and $0<u\in L^2(\Omega)$ such that $[T(t)u](x)>0$, that is, property \eqref{eq:zero-regular}.  We prove this property by comparison with a much easier case. Assume that $\omega$ is an open set with $\bar\omega\subseteq\Omega$. Let $(S(t))_{t\geq 0}$ denote the $C_0$-semigroup generated by the Dirichlet Laplacian on $L^2(\omega)$. Then
  \begin{equation}
    \label{eq:dirichlet-domination}
    0\leq S(t)u\leq T_D(t)\tilde u
  \end{equation}
  for all $0\leq u\in L^2(\omega)$, where $\tilde u$ is the extension of $u$ to $\Omega$ by zero. The proof of \eqref{eq:dirichlet-domination} is similar to that in \cite[Example~7.4.1(a)]{arendt:04:see} and can be omitted. This domination together with an explicit solution yields \eqref{eq:zero-regular} in the following way: Given $x_0=(x_{01},\dots,x_{0N})\in\Omega$ we choose a cube $\omega=x_0+(-r,r)^N$ with $r>0$ such that $\bar\omega\subseteq\Omega$. Set $\lambda_0:=\left(\pi/2r\right)^2$ and
  \begin{equation*}
    \varphi(x):=\prod_{k=1}^N\sin\left(\sqrt{\lambda_0}(x_k-x_{0k})\right)
  \end{equation*}
  for all $x=(x_1,\dots,x_N)\in\omega$. By an elementary calculation $S(t)\varphi=e^{-\lambda_0 t}\varphi$ for all $t\geq 0$. Clearly $\varphi(x_0)=1$ and thus by \eqref{eq:dirichlet-domination} we have $[T_D(t)\tilde\varphi](x_0)>0$ for all $t>0$.  As this applies to any choice of $x_0\in\Omega$ it follows that $(T_D(t))_{t\geq 0}$ satisfies \eqref{eq:zero-regular}.
\end{proof}

\subsection{The Laplacian with Robin boundary conditions}
\label{sec:robin-laplacian}
In this section we consider the Robin-Laplacian on a bounded, connected
and open set $\Omega\subseteq\mathbb R^N$ satisfying a very mild
condition at the boundary $\Gamma:=\partial\Omega$, namely
\begin{equation}
  \label{eq:Gamma-finite}
  \mathcal H^{N-1}(\Gamma)<\infty,
\end{equation}
where $\mathcal H^{N-1}$ is the $(N-1)$-dimensional Hausdorff measure. The Robin-Laplacian has been defined in \cite{daners:00:rbv} on arbitrary domains via the method of forms. The problem that occurs in this situation is that the natural form associated with the Robin-Laplacian is not closable in general, see \cite[Section~4]{arendt:03:lrb}. In \cite{daners:00:rbv} and later in \cite{arendt:03:lrb} this problem was solved by establishing a suitable closed form. Meanwhile, in \cite{arendt:12:sfd} it has been shown how a self-adjoint operator can be associated with any accretive, symmetric form even if it is not closable. We use this more recent approach here since it allows a description of the domain of the operator which realises in a precise way the Robin boundary condition
\begin{equation*}
  \frac{\partial u}{\partial\nu}+\beta u|_\Gamma=0,
\end{equation*}
where formally $\nu$ is the outer unit normal to $\Omega$. For the definition of the operator we need some preparation. The restriction $\sigma$ of $\mathcal H^{N-1}$ to the Borel sets of $\Gamma$ defines a Borel measure on $\Gamma$. Condition \eqref{eq:Gamma-finite} implies that $C(\Gamma)\subseteq L^2(\Gamma):=L^2(\Gamma,\sigma)$. Let $u\in H^1(\Omega)$. By
\begin{multline}
  \label{eq:trace-set}
  \Tr(u):=\bigl\{b\in L^2(\Gamma)\colon\text{There exist }u_n\in
  H^1(\Omega)\cap C(\bar\Omega)\\
  \text{with }u_n\to u\text{ in }H^1(\Omega)
  \text{ and }u_n|_\Gamma\to b$ in $L^2(\Gamma)\bigr\}
\end{multline}
we denote the set of all \emph{approximate traces} of $u$. If $\Omega$ has Lipschitz boundary, then for each $u\in H^1(\Omega)$ there exists exactly one $\tr(u)\in L^2(\Gamma)$ such that $\Tr(u)=\{\tr(u)\}$. In general $\Tr(u)$ may be empty or an infinite set. An extensive discussion and uniqueness criteria for the approximate trace are given in \cite{sauter:20:uat}. Next we define the normal derivative by means of Green's formula to be valid.

\begin{definition}[Normal derivative]
  \label{def:normal-derivative}
  Let $u\in H^1(\Omega)$ such that $\Delta u\in L^2(\Omega)$ in the sense of distributions. A function $b\in L^2(\Omega)$ is called the \emph{normal derivative of $u$} if
  \begin{equation}
    \label{eq:green-formula}
    \int_\Omega(\Delta u)v\,dx+\int_\Omega\nabla u\cdot\nabla
    v\,dx=\int_\Gamma bv\,d\sigma
  \end{equation}
  for all $v\in H^1(\Omega)\cap C(\bar\Omega)$. If such $b$ exists we write
  \begin{equation*}
    \partial_\nu u:=\frac{\partial u}{\partial\nu}:=b
  \end{equation*}
\end{definition}

Since $\sigma(\Gamma\cap B(z,r))>0$ for all $z\in\Gamma$ and $r>0$, the function $b$ in \eqref{eq:green-formula} is unique if it exists, so the above definition is justified.

Now let $\beta\in L^\infty(\Gamma)$ be such that $\beta(z)\geq\delta>0$ for all $z\in\Gamma$ for some constant $\delta>0$. We define the \emph{Laplacian $\Delta_\beta$ with Robin boundary conditions} or shortly the \emph{Robin Laplacian} as follows.
\begin{equation}
  \label{eq:robin-laplacian}
  \begin{aligned}
    D(\Delta_\beta) & :=\bigl\{u\in H^1(\Omega)\colon
    \Delta u\in L^2(\Omega),\exists u_\Gamma\in\Tr(u)\text{ with
    }\partial_\nu u+\beta u_\Gamma=0\bigr\},                           \\
    \Delta_\beta u  & :=\Delta u\qquad\text{for }u\in D(\Delta_\beta).
  \end{aligned}
\end{equation}
Since the normal derivative is unique and since $\beta(z)\neq 0$ for all
$z\in\Gamma$ there is at most one $u_\Gamma\in\Tr(u)$ such that
$\partial_\nu u+\beta u_\Gamma=0$. We have the following theorem.

\begin{theorem}[Robin Laplacian]
  \label{thm:robin-laplacian}
  The operator $\Delta_\beta$ is self-adjoint and generates a contractive $C_0$-semigroup $(T_\beta(t))_{t\geq 0}$ on $L^2(\Omega)$. This semigroup is holomorphic, positive, irreducible, ultra-contractive and locally smoothing.
\end{theorem}

In order to prove the ultra-contractivity we need the hypothesis that $\beta\geq\delta>0$, see the discussion on the Neumann Laplacian in Section~\ref{sec:neumann-laplacian} below, where $\beta=0$. Note however that we do not assume any regularity of the boundary $\Gamma$ of $\Omega$ besides the condition $\sigma(\Gamma)<\infty$. Even the latter condition is to keep the exposition simple and could be removed as done in \cite{arendt:03:lrb,daners:00:rbv}.

For the proof of Theorem~\ref{thm:robin-laplacian} we use the following generation theorem from \cite[Theorem~3.2 and Remark~3.5]{arendt:12:sfd}.

\begin{theorem}
  \label{thm:abstract-generation}
  Let $H$ be real Hilbert space, $D(\aaa)\subseteq H$ a dense subspace and let $\aaa\colon D(\aaa)\times D(\aaa)\to\mathbb R$ be bilinear. We assume that $\aaa$ is symmetric, that is, $\aaa(u,v)=\aaa(v,u)$ for all $u,v\in D(a)$ and accretive, that is, $\aaa(u,u)\geq 0$ for all $u\in D(\aaa)$.  Then there exists a unique self-adjoint operator $A$ on $H$ whose graph is given by the set of $(u,f)\in H\times H$ such that there exist $u_n\in D(\aaa)$ with $u_n\to u$ in $H$, $a(u_n-u_m)\to 0$ and $\aaa(u_n,v)\to\langle f,v\rangle$ for all $v\in D(\aaa)$ as $n,m\to\infty$ .
\end{theorem}

We call the operator $A$ given by the above theorem the \emph{operator associated with $\aaa$} and write $A\sim \aaa$. For simplicity we let $\aaa(u):=\aaa(u,u)$ for all $u\in D(\aaa)$.

\begin{remark}
  (a) The operator $A$ in Theorem~\ref{thm:abstract-generation} is single-valued.

  (b) The form $\aaa$ may not be \emph{closable}, that is, it might
  happen that there exist $u_n\in D(\aaa)$ such that $u_n\to 0$, $\aaa(u_n-u_m)\to 0$ as $n,m\to\infty$, but $\aaa(u_n)\not\to 0$ as $n\to\infty$. However, in that case there does not exist any $y\in H$ such that $\aaa(u_n,v)\to \langle y,v\rangle_H$ for all $v\in D(\aaa)$
\end{remark}

We will use Theorem~\ref{thm:abstract-generation} to prove the following proposition.

\begin{proposition}
  \label{prop:robin-laplacian}
  The operator $-\Delta_\beta$ is self-adjoint and monotone, that is, $\langle\Delta_\beta u,u\rangle_H\leq 0$ for all $u\in D(\Delta_\beta)$.
\end{proposition}
\begin{proof}
  Let $H=L^2(\Omega)$ and $D(\aaa)=H^1(\Omega)\cap C(\bar\Omega)$. Set
  \begin{equation}
    \label{eq:robin-form}
    \aaa(u,v)
    :=\int_\Omega \nabla u\cdot\nabla v\,dx+\int_\Gamma \beta uv\,d\sigma
  \end{equation}
  for all $u,v\in D(\aaa)$. Then $\aaa$ is bilinear, symmetric and accretive. Let $A\sim \aaa$. Then $A$ is self-adjoint and monotone. Thus $-A$ generates a contractive $C_0$-semigroup on $L^2(\Omega)$. We show that $A=-\Delta_\beta$.

  (a) We first prove that $A\subseteq-\Delta_\beta$.  Let $u\in D(A)$ with $Au=f$. Then there exist $u_n\in D(\aaa)$ such that $u_n\to u$ in $L^2(\Omega)$, $\aaa(u_n-u_m)\to 0$ as $n,m\to\infty$ and $\aaa(u_n,v)\to\int_\Omega fv\,dx$ as $n\to\infty$ for all $v\in D(A)$. It follows that $(u_n)$ is a Cauchy sequence in $H^1(\Omega)$ and $(u_n|_\Gamma)$ is a Cauchy sequence in $L^2(\Gamma)$, where the latter conclusion makes use of the fact that $\beta\geq \delta>0$. Thus $u\in H^1(\Omega)$ and $u_n\to u$ in $H^1(\Omega)$ as $n\to\infty$. Let
  \begin{equation*}
    u_\Gamma:=\lim_{n\to\infty}u_n|_\Gamma
  \end{equation*}
  in $L^2(\Gamma)$. Then $u_\Gamma\in\Tr(u)$ and for $v\in D(\aaa)$
  \begin{equation*}
    \begin{split}
      \int_\Omega fv\,dx&=\lim_{n\to\infty}
      \Bigl(\int_\Omega \nabla u_n\cdot\nabla v\,dx
      +\int_\Gamma \beta u_nv\,d\sigma\Bigr)\\
      &=\int_\Omega \nabla u\cdot\nabla v\,dx
      +\int_\Gamma \beta u_\Gamma v\,d\sigma.
    \end{split}
  \end{equation*}
  Choosing $v\in C_c^\infty(\Omega)$ we see in particular that $-\Delta_\beta=f$. Hence
  \begin{equation*}
    \int_\Omega \nabla u\cdot\nabla v\,dx+\int_\Omega(\Delta u)v\,dx
    =-\int_\Gamma \beta uv\,d\sigma.
  \end{equation*}
  for all $u\in D(\aaa)$. Thus $\partial_\nu u=-\beta u_\Gamma$ and thus $u\in D(\Delta_\beta)$ as claimed.

  (b) We next show that $-\Delta_\beta\subseteq A$. Let $u\in D(\Delta_\beta)$ and $-\Delta_\beta u=f$. Then $u\in H^1(\Omega)$, $\Delta u=f$ and there exists $u_\Gamma\in\Tr(u)$ such that $\partial_\nu u=-\beta u_\Gamma$. From the definition of $\Tr(u)$ we obtain $u_n\in D(\aaa)$ such that $u_n\to u$ in $H^1(\Omega)$ and $u_n|_{\Gamma}\to u$ in $L^2(\Gamma)$ as $n\to\infty$. Thus
  \begin{equation*}
    \aaa(u_n-u_m)
    =\int_\Omega |\nabla (u_n-u_m)|^2\,dx
    +\int_\Gamma \beta |u_n-u_m|^2\,d\sigma
    \to 0
  \end{equation*}
  as $n,m\to\infty$. Moreover, for $v\in D(\aaa)$
  \begin{multline*}
    \aaa(u_n,v)=\int_\Omega \nabla u_n\cdot\nabla v\,dx+\int_\Gamma \beta
    u_nv\,d\sigma\\
    \to \int_\Omega \nabla u\cdot\nabla v\,dx+\int_\Gamma \beta
    u_\Gamma v\,d\sigma
    =\int_\Omega fv\,dx
  \end{multline*}
  as $n\to\infty$ since $-\Delta u=f$ and $\partial_\nu u=\beta u_\Gamma$. Hence $u\in D(A)$.
\end{proof}

\begin{remark}
  The form $\aaa$ given by \eqref{eq:robin-form} is closable if and only if $\Tr(0)=\{0\}$, that is, if and only if each $u\in H^1(\Omega)$ has at most one approximate trace. A first example of non-uniqueness and hence non-closability of $\aaa$ was given in \cite[Examples~4.2 and~4.3]{arendt:03:lrb}, where $\Omega$ is an open and bounded subset of $\mathbb R^2$ which is not connected, and of an open, bounded and connected set in $\mathbb R^3$. Another example is given in \cite[Example~4.4]{arendt:11:dno}. Recently, a systematic study of the approximate trace appeared in \cite{sauter:20:uat}. There are some remarkable positive results; see \cite[Corollary~4.10 and~Theorem~5.3]{sauter:20:uat}:
  \begin{enumerate}[(a)]
  \item If $\Omega\subseteq\mathbb R^2$ is open, bounded and connected, then the approximate trace is unique.
  \item If $\Omega\subseteq\mathbb R^N$ ($N\geq 2$) is open, bounded and has a continuous boundary in the sense of graphs, then the approximate trace is unique. Having continuous boundary is equivalent to the \emph{segment condition} as shown in \cite[Theorem~5.4.4]{edmunds:18:std}.
  \end{enumerate}
\end{remark}

\begin{corollary}
  \label{cor:robin-laplacian}
  Let $u,f\in L^2(\Omega)$. Then the following assertions are equivalent.
  \begin{enumerate}[\upshape (i)]
  \item $u\in D(\Delta_\beta)$ and $-\Delta u=f$
  \item $u\in H^1(\Omega)$ and there exists a unique $u_\Gamma\in\Tr(u)$ such that
    \begin{equation}
      \label{eq:weak-robin}
      \int_\Omega \nabla u\cdot\nabla v\,dx+\int_\Gamma \beta
      u_\Gamma v\,d\sigma
      =\int_\Omega fv\,dx
    \end{equation}
    for all $v\in C(\bar\Omega)\cap H^1(\Omega)$.
  \item $u\in H^1(\Omega)$ and there exist $u_n\in H^1(\Omega)\cap C(\bar\Omega)$ such that $u_n\to u$ in $H^1(\Omega)$, $\lim_{n\to\infty}u_n|_{\Gamma}$ exist in $L^2(\Gamma)$ and $\aaa(u_n,v)\to\int_\Omega fv\,dx$ for all $v\in D(\aaa)$ as $n\to\infty$, where $\aaa$ is given by \eqref{eq:robin-form}.
  \end{enumerate}
  In that case $u_\Gamma=\lim_{n\to\infty}u_n$ in $L^2(\Gamma)$ for every sequence $(u_n)$ from (iii) and for each $u_\Gamma$ given by (ii).
\end{corollary}
\begin{proof}
  (i) $\iff$ (ii): This follows from the definition of $D(\Delta_\beta)$.

  (iii) $\implies$ (ii):  This is part (a) of the proof of Proposition~\ref{prop:robin-laplacian}.

  (ii) $\implies$ (iii): This follows from part (b) in the proof of Proposition~\ref{prop:robin-laplacian}.
\end{proof}

In a series of propositions we now prove the properties of the $C_0$-semigroup $(T_\beta(t))_{t\geq 0}$ generated by $\Delta_\beta$ as stated in Theorem~\ref{thm:robin-laplacian}.

\begin{proposition}[Positivity]
  \label{prop:robin-sg-positive}
  We have $T_\beta(t)\geq 0$ for all $t\geq 0$.
\end{proposition}
\begin{proof}
  The statement is equivalent to $(\lambda I-\Delta_\beta)^{-1}\geq 0$ for all $\lambda>0$. Let $\lambda>0$, $f\in L^2(\Omega)$ and set $u:=(\lambda I-\Delta_\beta)^{-1}f$. Assuming that $f\leq 0$ we have to show that $u\leq 0$. Since $u\in D(\Delta_\beta)$ and $\lambda   u-\Delta_\beta u=f$, by Corollary~\ref{cor:robin-laplacian} there exists a unique $u_\Gamma\in L^2(\Gamma)$ such that
  \begin{equation}
    \label{eq:robin-weak-resolvent}
    \lambda\int_\Omega uv\,dx+\int_\Omega\nabla u\cdot\nabla
    v\,dx+\int_\Gamma\beta u_\Gamma v\,d\sigma
    =\int_\Omega fv\,dx
  \end{equation}
  for all $v\in H^1(\Omega)\cap C(\bar\Omega)$. Since $u_\Gamma\in\Tr(u)$ there exist $u_n\in H^1(\Omega)\cap C(\bar\Omega)$ such that $u_n\to u$ in $H^1(\Omega)$ and $u_n|_\Gamma\to u_\Gamma$ in $L^2(\Gamma)$. Then also  $u_n^+\in H^1(\Omega)\cap C(\bar\Omega)$, $u_n^+\to u^+$ in $H^1(\Omega)$ and $u_n^+|_\Gamma\to u_\Gamma^+$ in $L^2(\Gamma)$. Letting $v=u_n^+$ in \eqref{eq:robin-weak-resolvent} gives
  \begin{equation*}
    \lambda\int_\Omega uu_n^+\,dx+\int_\Omega\nabla u\cdot\nabla u_n^+\,dx
    +\int_\Gamma\beta uu_n^+\,d\sigma
    \leq 0
  \end{equation*}
  for all $n\in\mathbb N$. Sending $n\to\infty$ and using that $\nabla   u_n^+=1_{\{u_n>0\}}\nabla u_n$ we obtain
  \begin{equation*}
    \lambda\int_\Omega |u^+|^2\,dx+\int_\Omega|\nabla u^+|^2\,dx
    +\int_\Gamma\beta |u^+|^2\,d\sigma
    \leq 0.
  \end{equation*}
  As $\lambda,\beta>0$ this implies that $u^+=0$, that is, $u\leq 0$.
\end{proof}

Some properties of the Robin Laplacian can be obtained from the Dirichlet Laplacian $\Delta_D$ given in Subsection~\ref{sec:dirichlet-laplacian}. Denote the $C_0$-semigroup generated by $\Delta_D$ on $L^2(\Omega)$ by $(T_D(t))_{t\geq 0}$.

\begin{proposition}[Domination]
  \label{prop:robin-domination}
  One has $0\leq T_D(t)\leq T_\beta(t)$ for all $t>0$.
\end{proposition}
\begin{proof}
  We know already that $T_D(t)\geq 0$. In order to show that $T_D(t)\leq T_\beta(t)$ for all $t>0$ it suffices to show that
  \begin{equation*}
    (\lambda I-\Delta_D)^{-1}\leq (\lambda I-\Delta_\beta)^{-1}
  \end{equation*}
  whenever $\lambda>0$. Let $0\leq f\in L^2(\Omega)$ and set $u:=(\lambda I-\Delta_\beta)^{-1}f$ and $w:=(\lambda I-\Delta_D)^{-1}f$. We know from Subsection~\ref{sec:dirichlet-laplacian} that $w\geq 0$. We also know from Proposition~\ref{prop:robin-sg-positive} that $u\geq 0$. We have to show that $w\leq u$. By definition $w\in H_0^1(\Omega)$ and
  \begin{equation}
    \label{eq:dirichlet-weak-resolvent}
    \lambda\int_\Omega wv\,dx+\int_\Omega\nabla w\cdot\nabla v
    =\int_\Omega fv\,dx
  \end{equation}
  for all $v\in H_0^1(\Omega)$. Moreover, $u\in H^1(\Omega)$ and there exists a unique $u_\Gamma\in\Tr(u)$ such that \eqref{eq:robin-weak-resolvent} holds for all $v\in H^1(\Omega)\cap C(\bar\Omega)$. Subtracting \eqref{eq:dirichlet-weak-resolvent} from \eqref{eq:robin-weak-resolvent} we obtain
  \begin{equation}
    \label{eq:resolvent-difference}
    \lambda\int_\Omega (w-u)v\,dx+\int_\Omega\nabla (w-u)\cdot\nabla v\,dx
    -\int_\Gamma\beta u_\Gamma v\,d\sigma
    =0
  \end{equation}
  for all $v\in H_0^1(\Omega)\cap C(\bar\Omega)$. Since $w\in H_0^1(\Omega)$ there exist $w_n\in C_c^\infty(\Omega)$ such that $w_n\to w$ in $H_0^1(\Omega)$. Then $w_n^+\in H^1(\Omega)\cap C_0(\Omega)$ and $w_n^+\to w^+$ in $H^1(\Omega)$ as $n\to\infty$. As $w\geq 0$ we have $w=w^+$. Since $u_\Gamma\in\Tr(u)$ there exist $u_n\in H^1(\Omega)\cap C(\bar\Omega)$ such that $u_n\to u$ in $H^1(\Omega)$ and $u_n|_\Gamma\to u_\Gamma$ in $L^2(\Gamma)$ as $n\to\infty$. Since $u\geq 0$, replacing $u_n$ by $u_n^+$ we may assume without loss of generality that $u_n\geq 0$. Then $v_n:=(w_n-u_n)^+\in H^1(\Omega)\cap C_0(\Omega)$ and $v_n\to (w-u)^+$ in $H^1(\Omega)$. Putting $v=v_n$ in \eqref{eq:resolvent-difference} and sending $n\to\infty$ yields
  \begin{equation}
    \label{eq:resolvent-difference-limit}
    \lambda\int_\Omega |(w-u)^+|^2\,dx+\int_\Omega|\nabla (w-u)^+|^2\,dx
    =0,
  \end{equation}
  where we used that $v_n|_\Gamma=0$, $(w-u)(w-u)^+=|(w-u)^+|^2$ and $\nabla(w-u)^+=1_{\{w>u\}}\nabla(w-u)$ so that $\nabla(w-u)\cdot\nabla(w-u)^+=|\nabla(w-u)^+|^2$. Hence \eqref{eq:resolvent-difference-limit} shows that $(w-u)^+=0$, that is, $w\leq u$ as claimed.
\end{proof}

\begin{proposition}
  \label{prop:robin-locally-smoothing}
  The semigroup $(T_\beta(t))_{t\geq 0}$ is locally smoothing.
\end{proposition}
\begin{proof}
  (a) We know already that $(T_\beta(t))_{t\geq 0}$ is positive and holomorphic. Hence $T_\beta(t)L^2(\Omega)\subseteq D((\Delta_\beta)^k)$ for all $k\in\mathbb N$. For $u\in D((\Delta_\beta)^k)$ we have $\Delta^k u\in L^2(\Omega)$ for all $k\in\mathbb N$. Now apply the same arguments as in the proof of Theorem~\ref{thm:dirichlet-laplacian} to see that $T(t)L^2(\Omega)\subseteq C(\Omega)$ for all $t>0$.

  (b) We know from Subsection~\ref{sec:dirichlet-laplacian} that $(T_D(t))_{t\geq 0}$ satisfies \eqref{eq:zero-regular}. Thus, given $x\in \Omega$ there exists $w\in L^2(\Omega)$ and $t>0$ such that $(T_D(t)w)(x)\neq 0$. By Proposition~\ref{prop:robin-domination}
  \begin{equation*}
    (T_\beta(t)|w|)(x)
    \geq(T_\beta(t)|w|)(x)
    \geq\bigl|(T_D(t)w)(x)\bigr|
    >0
  \end{equation*}
  and thus also $(T_\beta(t))_{t\geq 0}$ satisfies \eqref{eq:zero-regular}.
\end{proof}

We next show that the semigroup $(T_\beta(t))_{t\geq 0}$ is \emph{submarkovian}, that is, $T_\beta(t)1_\Omega\leq 1_\Omega$ for all $t\geq 0$.

\begin{proposition}
  \label{prop:robin-submarkovian}
  The semigroup $(T_\beta(t))_{t\geq 0}$ is submarkovian.
\end{proposition}
\begin{proof}
  Let $\lambda>0$ and set $u:=(\lambda I-\Delta_\beta)^{-1}1_\Omega$. We know from Proposition~\ref{prop:robin-sg-positive} that $u\geq 0$. If we can show that $\lambda u\leq 1_\Omega$, then $\bigl(\lambda(\lambda I-\Delta_\beta)^{-1}\bigr)^n1_\Omega\leq 1_\Omega$ for all $n\in\mathbb N$. Thus, setting $\lambda=n/t$ and taking limits,
  \begin{equation*}
    T_\beta(t)=\lim_{n\to\infty}
    \Bigl(\frac{n}{t}\Bigl(\frac{n}{t}-\Delta_\beta\Bigr)^{-1}\Bigr)^n1_\Omega\leq
    1_\Omega
  \end{equation*}
  for all $t>0$. It remains to show that $\lambda u\leq 1_\Omega$. By definition of $\Delta_\beta$ there exists $u_\Gamma\in\Tr(u)$ such that
  \begin{equation*}
    \lambda\int_\Omega uv\,dx+\int_\Omega\nabla u\nabla
    v\,dx+\int_\Gamma u_\Gamma v\,d\sigma
    =\int_\Omega 1 v\,dx
  \end{equation*}
  and thus
  \begin{equation}
    \label{eq:robin-torsion}
    \int_\Omega (\lambda u-1)v\,dx+\int_\Omega\nabla u\nabla
    v\,dx+\int_\Gamma u_\Gamma v\,d\sigma
    =0
  \end{equation}
  for all $v\in H^1(\Omega)\cap C(\bar\Omega)$. Since $u_\Gamma\in\Tr(u)$ there exist $0\leq u_n\in H^1(\Omega)\cap C(\bar\Omega)$ such that $u_n\to u$ in $H^1(\Omega)$ and $u_n|_\Gamma\to u_\Gamma$ in $L^2(\Gamma)$ as explained in the proof of Proposition~\ref{prop:robin-locally-smoothing}. In particular $u_\Gamma\geq 0$. Moreover, $(\lambda u-1)^+=\lim_{n\to\infty}v_n$ in $H^1(\Omega)$ if we set $v_n:=(\lambda u_n-1)^{+}\in H^1(\Omega)\cap C(\bar\Omega)$. Applying \eqref{eq:robin-torsion} to $v=v_n$ and using $u_\Gamma\geq 0$ we obtain
  \begin{equation*}
    \int_\Omega (\lambda u-1)v_n\,dx
    +\int_\Omega\nabla u\nabla v_n\,dx
    \leq 0
  \end{equation*}
  for all $n\in\mathbb N$. Sending $n\to\infty$ yields
  \begin{equation*}
    \int_\Omega |(\lambda u-1)^+|^2\,dx
    +\int_{[\lambda u\geq 1]}|\nabla u|^2\,dx
    \leq 0.
  \end{equation*}
  Here we use that $\nabla(\lambda u-1)^+=\lambda 1_{\{\lambda u>1\}}\nabla u$ and thus
  \begin{equation*}
    \int_\Omega \nabla u\nabla(\lambda u-1)^+=\lambda\int_{[\lambda
          u>1]}|\nabla u|^2\geq 0
  \end{equation*}
  Hence $(\lambda u-1)^+=0$, that is, $\lambda u\leq 1_\Omega$ as claimed.
\end{proof}

We finally show that $(T_\beta(t))_{t\geq 0}$ is ultra-contractive. This has been discovered in \cite{daners:00:hke} as a consequence of Mazya's inequality. See also \cite{arendt:03:lrb} for an alternative proof based on the notion of relative capacity. The aim of this exposition is to provide a more elementary proof.

A \emph{submarkovian symmetric semigroup} on $L^2(\Omega)$ is a $C_0$-semigroup $(S(t))_{t\geq 0}$ such that $S(t)=S(t)^*$, $S(t)\geq 0$ and $S(t)1_\Omega\leq 1_\Omega$ for all $t>0$. We note that \cite{davies:89:hks} uses the term ``markovian'' instead of submarkovian. Denote by $-A$ the generator of this semigroup on $L^2(\Omega)$. There exists a unique closed, accretive symmetric form $\aaa_c$ on $L^2(\Omega)$ such that $A$ is associated with $\aaa_c$. Ultra-contractivity with polynomial decay can be characterised in terms of this form $\aaa_c$ as shown in \cite[Theorem~2.4.2]{davies:89:hks}. Note that in \cite{davies:89:hks} a positive form is the same as an accretive, symmetric form, a notion we prefer here to avoid confusion with positivity in the lattice sense as we use it throughout, and also since we restrict ourselves to the real case. In fact, only in the complex case $\aaa(u)\in[0,\infty)$ for all $u\in D(\aaa)$ implies that $\aaa$ is symmetric.

\begin{theorem}[ultra-contractive semigroups]
  \label{thm:ultra-contractive}
  Let $\mu>2$. Then the following assertions are equivalent.
  \begin{enumerate}[\upshape (i)]
  \item There exists $c>0$ such that $\|S(t)u\|_\infty\leq ct^{-\mu/4}\|u\|_2$ for all $u\in L^2(\Omega)$ and all $t>0$.
  \item There exists $c>0$ such that $\|u\|_{2\mu/(\mu-2)}^2\leq   c\aaa_c(u)$ for all $u\in D(\aaa_c)$.
  \end{enumerate}
\end{theorem}

One form of Mazya's inequality is given by
\begin{equation}
  \label{eq:mazya-inequality}
  \|u\|_{2N/(N-1)}^2\leq c\Bigl(\int_\Omega|\nabla u|^2\,dx+\int_\Gamma|u^2|\,d\sigma\Bigr)
\end{equation}
for all $u\in D(\aaa)=H^1(\Omega)\cap C(\bar\Omega)$ and some $c>0$ only depending on $N$ and the measure of $\Omega$, see \cite[Corollary~4.11.1/2]{mazya:85:ssp}. Note that $2N/(N-1)=2\mu/(\mu-2)$ if we set $\mu:=2N$. If $N=1$ we can choose any $\mu>2$ arbitrary since in this case any connected open set is an interval and hence smooth. Since $\beta\geq \delta>0$ it follows that there exists $c_1>0$ such that
\begin{equation}
  \label{eq:mazya-inequality-form}
  \|u\|_{2\mu/(\mu-2)}^2\leq c_1\aaa(u)
\end{equation}
for all $u\in D(\aaa)$. In general, the form with domain $D(\aaa)$ is not closed, and not even closable as mentioned before. We have to show \eqref{eq:mazya-inequality-form} with $\aaa$ replaced by $\aaa_c$. The inequality is not obvious since $\aaa_c(u)\leq \aaa(u)$ for all $u\in H^1(\Omega)\cap C(\bar\Omega)$. If $\aaa$ is not closed we even have $\aaa_c(u)<\aaa(u)$ for some $u\in H^1(\Omega)\cap C(\bar\Omega)$. For this reason we now give a description of $\aaa_C$.

If $S\subseteq\Gamma$ is a Borel set we let
\begin{equation*}
  L^2(\Gamma,S):=\{w\in L^2(\Gamma)\colon\text{$w=0$ $\sigma$-almost
    everywhere on $\Gamma\setminus S$}\}.
\end{equation*}

\begin{proposition}
  \label{prop:trace}
  There exists a Borel set $\Gamma_s\subseteq\Gamma$ such that $\Tr(0)=L^2(\Gamma,\Gamma_s)$, where $0\in H^1(\Omega)$ is the constant function with value zero.
\end{proposition}
\begin{proof}
  It follows from the definition that $\Tr(0)$ is a closed subspace of $L^2(\Gamma)$. By \cite[Lemma~3.4]{daners:00:rbv} one has $L^\infty(\Gamma)\Tr(0)\subseteq \Tr(0)$. Thus $\Tr(0)$ is a closed lattice ideal of $L^2(\Gamma)$. Now the claim follows from \cite[Section~III.1, Example~2]{schaefer:74:blp}.
\end{proof}

We call $\Gamma_s$ in the above proposition the \emph{singular part} of $\Gamma$ and $\Gamma_r:=\Gamma\setminus\Gamma_s$ the \emph{regular part}. Next we show that the traces occurring in the definition of $\Delta_\beta$ always live on the regular part.

\begin{lemma}
  \label{lem:domain-trace}
  Let $u\in D(\Delta_\beta)$, that is, $u\in H^1(\Omega)$, $\Delta u\in
    L^2(\Omega)$ and there exists a unique $u_\Gamma\in\Tr(u)$ such that
  $\partial_\nu u+\beta u=0$. Then $u_\Gamma\in L^2(\Gamma,\Gamma_r)$.
\end{lemma}
\begin{proof}
  We want to show that $\int_\Gamma u_\Gamma w\,d\sigma = 0$ or equivalently
  \begin{equation}
    \label{eq:trace-integral}
    \int_\Gamma \beta u_\Gamma w\,d\sigma = 0
  \end{equation}
  for all $w\in L^2(\Gamma,\Gamma_s)$. Let $w\in L^2(\Gamma,\Gamma_s)$. By Proposition~\ref{prop:trace} we know that $w\in\Tr(0)$ and thus there exist $v_n\in H^1(\Omega)\cap C(\bar\Omega)$ such that $\|w-v_n\|_{L^2(\Gamma)}<1/n$ and $\|v_n\|_{H^1}<1/n$ for all $n\in\mathbb N$. We have
  \begin{equation*}
    \int_\Omega(\Delta u)v_n\,dx+\int_\Omega\nabla u\cdot\nabla v_n\,dx
    =-\int_\Gamma\beta u_\Gamma v_n\,d\sigma
  \end{equation*}
  for all $n\in\mathbb N$. Sending $n\to\infty$ we obtain \eqref{eq:trace-integral}.
\end{proof}

With the help of the Borel set $\Gamma_r$ we define the form $\aaa_r$ on $L^2(\Omega)$ by $D(\aaa_r)=H^1(\Omega)\cap C(\bar\Omega)$ and
\begin{equation}
  \label{eq:robin-form-regular}
  \aaa_r(u,v):=
  \int_\Omega\nabla u\nabla v\,dx+\int_{\Gamma_r}\beta uv\,d\sigma
\end{equation}
for all $u,v\in D(\aaa_r)$. The form $\aaa_r$ is obviously accretive and symmetric.

\begin{proposition}
  \label{prop:robin-characterisation}
  The form $\aaa_r$ defined above is closable and its closure $\bar\aaa_r$ is the unique closed accretive and symmetric form such that $-\Delta_\beta\sim\bar\aaa_r$.
\end{proposition}
\begin{proof}
  We first show that $\aaa_r$ is closable. Let $u_n\in H^1(\Omega)\cap C(\bar\Omega)$ such that $u_n\to 0$ in $L^2(\Omega)$ and $\aaa_r(u_n-u_m) \to 0$ as $n,m\to\infty$. Then $(u_n)$ is a Cauchy sequence in $H^1(\Omega)$ and $u_n\to 0$ in $H^1(\Omega)$. Moreover, as $\beta\geq\delta>0$, we also have that $u_n|_{\Gamma_r}\to b$ in $L^2(\Gamma)$ for some $b\in L^2(\Gamma,\Gamma_r)$. Since by Proposition~\ref{prop:trace} $u_n1_{\Gamma_s}\in L^2(\Gamma,\Gamma_s)=\Tr(0)$, there exists $w_n\in H^1(\Omega)\cap C(\bar\Omega)$ such that $\|w_n\|_{H^1}<1/n$ and $\|u_n1_{\Gamma_s}-w_n\|_{L^2(\Gamma)}<1/n$ for all $n\in\mathbb N$. Thus $v_n:=u_n-w_n\in H^1(\Omega)\cap C(\bar\Omega)$ and $v_n\to 0$ in $H^1(\Omega)$ as $n\to\infty$. Let $\tilde b\in L^2(\Gamma)$ with $\tilde b(z)=b(z)$ for $z\in\Gamma_r$ and $\tilde b(z)=0$ for $z\in\Gamma_s$. Using that $\Gamma=\Gamma_r\cup\Gamma_s$ is a disjoint union, we deduce that
  \begin{equation*}
    \begin{split}
      \|v_n-\tilde b\|_{L^2(\Gamma)}
      &=\|u_n-\tilde b-w_n\|_{L^2(\Gamma_r)}+\|u_n-w_n\|_{L^2(\Gamma_s)}\\
      &\leq\|u_n-\tilde b\|_{L^2(\Gamma_r)}+\|w_n\|_{L^2(\Gamma_r)}+\|u_n-w_n\|_{L^2(\Gamma_s)}\\
      &=\|u_n-\tilde b\|_{L^2(\Gamma_r)}+\|u_n1_{\Gamma_s}-w_n\|_{L^2(\Gamma_s)}
      \to 0
    \end{split}
  \end{equation*}
  as $n\to\infty$. Therefore $\tilde b\in\Tr(0)\cap L^2(\Gamma_r)=L^2(\Gamma_s)\cap   L^2(\Gamma_r)=\{0\}$. This shows that $u_n|_{\Gamma_r}\to 0$ in $L^2(\Gamma_r)$ and so $\aaa_r$ is closable.

  We next show that $-\Delta\sim\aaa_r$. Let $A$ be the operator on $L^2(\Omega)$ associated with $\aaa_r$. We show that $-\Delta_\beta\subseteq A$. Since $I+A\colon D(A)\to L^2(\Omega)$ and $I-\Delta_\beta\colon D(\Delta_\beta)\to L^2(\Omega)$ are both injective this implies that $-\Delta_\beta=A$.

  Let $u\in D(\Delta_\beta)$ and $-\Delta_\beta u=f$. Then by Corollary~\ref{cor:robin-laplacian}  we know that $u\in H^1(\Omega)$, $u_\Gamma\in L^2(\Gamma)$ and there exist $u_n\in H^1(\Omega)\cap C(\bar\Omega)$ such that $u_n\to u$ in $H^1(\Omega)$, $u_n|_{\Gamma}\to u_\Gamma$ and
  \begin{equation*}
    \int_\Omega \nabla u_n\cdot\nabla v\,dx+\int_{\Gamma} \beta u_nv\,d\sigma
    \to\int_\Omega fv\,dx
  \end{equation*}
  as $n\to\infty$ for all $v\in H^1(\Omega)\cap C(\bar\Omega)$. By Lemma~\ref{lem:domain-trace} one has $u_{\Gamma}\in   L^2(\Gamma,\Gamma_r)$. Thus
  \begin{equation*}
    \aaa_r(u_n,v)\to\int_\Omega fv\,dx
  \end{equation*}
  as $n\to\infty$ for all $v\in H^1(\Omega)\cap C(\bar\Omega)$. Hence, by definition of $A$ we have $u\in D(A)$ and $Au=f$.
\end{proof}

We can finally prove the ultra-contractivity of $(T_\beta(t))_{t\geq 0}$.

\begin{proposition}
  Let $(T_\beta(t))_{t\geq 0}$ be the $C_0$-semigroup generated by $-\Delta$. Let $\mu=2N$ if $N\geq 2$ and $\mu\in(2,\infty)$ arbitrary for $N=1$. Then there exists a constant $c>0$ such that
  \begin{equation*}
    \|T_\beta(t)u\|_\infty\leq ct^{-\mu/4}\|u\|_2,
  \end{equation*}
  for all $u\in L^2(\Omega)$ and all $t>0$.
\end{proposition}
\begin{proof}
  By \eqref{eq:mazya-inequality}
  \begin{equation*}
    \|u\|_{2\mu/(\mu-2)}^2\leq c_1\aaa(u)
    =c_1\Bigl(\int_\Omega|\nabla u|^2\,dx + \int_\Gamma|u|^2\,d\sigma\Bigr)
  \end{equation*}
  for all $u\in H^1(\Omega)\cap C(\bar\Omega)$. Given $u\in H^1(\Omega)\cap C(\bar\Omega)$ there exist $w_n\in H^1(\Omega)\cap C(\bar\Omega)$ such that $w_n\to 0$ in $H^1(\Omega)$ and $w_n|_\Gamma\to u1_{\Gamma_r}$ in $L^2(\Gamma)$ and $\sigma$-almost everywhere. Then $u_n:=u-w_n\in H^1(\Omega)\cap C(\bar\Omega)$ with $u_n\to u$ in $H^1(\Omega)$ and almost everywhere. Moreover $u_n|_{\Gamma}\to u1_{\Gamma_r}$ in $L^2(\Gamma)$.

  Applying \eqref{eq:mazya-inequality-form} to $u_n$ we obtain by Fatou's Lemma
  \begin{multline*}
    \|u\|_{2\mu/(\mu-2)}^2
    \leq c_1\liminf_{n\to\infty}\Bigl(\int_\Omega|\nabla u_n|^2\,dx
    + \int_\Gamma|u_n|^2\,d\sigma\Bigr)\\
    =c_1\Bigl(\int_\Omega|\nabla u|^2\,dx
    + \int_{\Gamma_r}|u_n|^2\,d\sigma\Bigr)
    =c_1\aaa_r(u).
  \end{multline*}
  Now the claim follows from Theorem~\ref{thm:ultra-contractive}.
\end{proof}
This finally completes the proof of Theorem~\ref{thm:robin-laplacian}.

\subsection{The Laplacian with Neumann boundary conditions}
\label{sec:neumann-laplacian}
Let again $\Omega\subseteq\mathbb R^N$ be an open, bounded and connected set. The \emph{Laplacian with Neumann boundary conditions} or briefly the \emph{Neumann Laplacian} is the operator $\Delta_N$ on $L^2(\Omega)$ defined by
\begin{equation}
  \label{eq:neumann-laplacian}
  \begin{aligned}
    D(\Delta_N) & =\{u\in H^1(\Omega)\colon\Delta u\in
    L^2(\Omega),\partial_\nu u=0\},                              \\
    \Delta_Nu   & :=\Delta u\qquad\text{for $u\in D(\Delta_N)$.}
  \end{aligned}
\end{equation}
Here the outer normal derivative $\partial_\nu$ is to be interpreted in a
generalised sense.  For $u\in H^1(\Omega)$ such that
$\Delta u\in L^2(\Omega)$ we say that $\partial_\nu u=0$ if
\begin{equation*}
  \int_\Omega(\Delta u)v\,dx+\int_\Omega\nabla u\cdot\nabla v\,dx=0
  \text{ for all }v\in H^1(\Omega).
\end{equation*}
This is consistent with Definition~\ref{def:normal-derivative} if $\mathcal   H^{N-1}(\Gamma)<\infty$ and $H^1(\Omega)\cap C(\bar\Omega)$ is dense in $H^1(\Omega)$. This is for instance the case if $\Omega$ has continuous boundary.  It is easy to see that $-\Delta_N$  is associated with the closed, accretive and symmetric form on $L^2(\Omega)$ given by
\begin{equation}
  \label{eq:neumann-form}
  \aaa(u,v):=\int_{\Omega}\nabla u\cdot\nabla v\,dx
\end{equation}
for all $u,v\in D(\aaa):=H^1(\Omega)$. Thus $\Delta$ is self-adjoint and dissipative and hence generates a contractive $C_0$-semigroup $(T_N(t))_{t\geq 0}$ on $L^2(\Omega)$. This semigroup is positive and irreducible since in particular
\begin{equation}
  \label{eq:neumann-domination}
  0\leq T_D(t)\leq T_N(t)
\end{equation}
for all $t>0$ and $(T_D(t))_{t\geq 0}$ is irreducible. The proof of \eqref{eq:neumann-domination} is very similar to that of \eqref{eq:dirichlet-domination}. Again, the local smoothing property can be proved in the same way as in Theorem~\ref{thm:dirichlet-laplacian}. Elliptic regularity implies that $T(t)L^2(\Omega)\subseteq C(\Omega)$ and the domination property \eqref{eq:neumann-domination} implies \eqref{eq:zero-regular}.

The last property we need is the ultra-contractivity. This cannot be expected for general $\Omega$, so we need to require some regularity of the boundary. We cannot directly apply Theorem~\ref{thm:ultra-contractive}, but we can apply it to the form
\begin{equation*}
  \aaa_1(u,v):=\int_\Omega\nabla u\cdot\nabla v\,dx+\int_\Omega uv\,dx
\end{equation*}
on $D(\aaa_1)=H^1(\Omega)$. This is the closed form associated with $I-\Delta_N$. Moreover, $\aaa_1(u)=\|u\|_{H^1}^2$. The corresponding semigroup is given by $e^{-t}T_N(t)$ for all $t>0$. If $\Omega$ allows an embedding $H^1(\Omega)\subseteq L^p(\Omega)$ for some $p>2$, then there exists a constant $c>0$ such that
\begin{equation*}
  \|u\|_{2\mu/(\mu-2)}^2\leq c\aaa_1(u)
\end{equation*}
for all $u\in D(\aaa_1)$. Indeed, a simple computation shows that $\mu=2p/(p-2)>2$. It follows that there exists $c_1>0$ such that
\begin{equation*}
  \|T_N(t)u\|_\infty\leq c_1e^{t}t^{-\mu}\|u\|_2
\end{equation*}
for all $u\in L^2(\Omega)$, which implies that $(T_N(t))_{t\geq 0}$ is ultra-contractive.

If $\Omega$ is a Lipschitz domain, then the standard Sobolev inequality applies and thus $(T_N(t))_{t\geq 0}$ is ultra-contractive by the above reasoning. The same remains true if $\Omega$ satisfies a uniform interior cone condition, see \cite[Theorem~4.4 and~5.2]{daners:00:hke}. Other examples include domains with outward pointing cusps of polynomial order such as discussed in \cite[Theorem~5.35]{adams:75:ssp} with $p>N$ depending on the sharpness of the cusp. We summarise our discussion in the following theorem.

\begin{theorem}[Neumann Laplacian]
  Assume that $\Omega\subseteq\mathbb R^N$ is open, bounded, connected. Then the $C_0$-semigroup $(T_N(t))_{t\geq 0}$ generated by the Neumann Laplacian is positive, irreducible, holomorphic, and locally smoothing.  If $\Omega$ allows an embedding $H^1(\Omega)\subseteq L^p(\Omega)$ for some $p>2$, then $(T_N(t))_{t\geq 0}$ is also ultra-contractive.
\end{theorem}

\subsection{The semilinear problem for the Laplacian}
\label{sec:semilinear-laplacian}
Let $\Omega\subseteq\mathbb R^N$ be bounded, open and connected. Let $g$ be a function satisfying the conditions of Section~\ref{sec:logistic}. Let $A$ be one of the three operators
\begin{enumerate}[(a)]
\item $A=-\Delta_D$;
\item $A=-\Delta_\beta$, where $\beta\in L^\infty(\Gamma)$,
  $\beta\geq\delta$ for some constant $\delta>0$,
  $\Gamma=\partial\Omega$, and where we assume that
  $\mathcal H^{N-1}(\Gamma)<\infty$;
\item $A=-\Delta_N$ under assumptions that guarantee that
  $H^1(\Omega)\subseteq L^p(\Omega)$ for some $p>2$ (in particular if
  $\Omega$ satisfies a uniform interior cone condition).
\end{enumerate}
We let $0<m\in L^\infty(\Omega)$ and that $\lambda\in\mathbb R$ satisfies the conditions of Theorem~\ref{thm:main}. Then the problem \eqref{eq:logeq} has a unique non-trivial positive solution.

When dealing with semilinear problems one is often interested in \emph{weak solutions}. Let $-A$ be one of the realisations of the Laplace operators considered above. As we have seen, this operator is associated with a closed symmetric and accretive form $\aaa$ with some domain $D(\aaa)$. Let $V=D(\aaa)$ with norm
\begin{equation*}
  \|u\|_V:=\bigl(\|u\|_{L^2(\Omega)}^2+\aaa(u)\bigr)^{1/2}.
\end{equation*}
Then $V$ is a Hilbert space with $V\hookrightarrow H$ which is dense in $H$.  Let $\tilde A\in\mathcal L(V,V')$ be defined by $\langle\tilde Au,v\rangle_{V',V}:=\aaa(u,v)$ for all $u,v\in V$. Then $(\omega I+\tilde A)^{-1}\in\mathcal L(V',V)$ for all $\omega>0$. We call $u\in V$ a \emph{weak solution} of \eqref{eq:logeq} if $mg(\cdot\,,u)u\in V'$ and
\begin{equation*}
  \aaa(u,v)=\lambda\int_\Omega uv\,dx-\int_\Omega mg(\cdot\,,u)uv\,dx
\end{equation*}
for all $v\in D(\aaa)$, or equivalently
\begin{equation*}
  \tilde Au=\lambda u-mg(\cdot\,,u)u\in V'.
\end{equation*}
If we choose $\omega>0$, then the fixed point equation \eqref{eq:fixed-pt-eq} is still valid with $A$ replaced by $\tilde A$ and hence
\begin{equation*}
  u
  =(\lambda+\omega)(\omega I+A)^{-1}u
  -(\omega I+\tilde A)^{-1}\bigl(mg(\cdot\,, u)u\bigr)
  \leq(\lambda+\omega)(\omega I+A)^{-1}u,
\end{equation*}
showing that \eqref{fixed-pt-ineq} still applies. Thus the proof of Proposition~\ref{prop:solution-bounded} shows that any weak solution $u$ is in $L^\infty(\Omega)$. As a consequence, $mg(\cdot\,,u)u\in L^2(\Omega)$, which in turn implies that $u\in D(A)$ by \eqref{eq:fixed-pt-eq} and thus $u$ is a solution of \eqref{eq:logeq} in the sense of Section~\ref{sec:logistic}. Hence all results in Section~\ref{sec:logistic} apply to weak solutions.

\section{Elliptic boundary value problems in divergence form}
\label{sec:elliptic-divergence-form}
In this section we briefly outline how the abstract results in this paper apply to more general elliptic operators, not necessarily self-adjoint. This includes the remarks in Subsection~\ref{sec:semilinear-laplacian} on weak solutions to the abstract logistic equation.

Suppose that $\Omega\subseteq\mathbb R^N$ is a bounded domain and that $\mathcal A$ is an elliptic operator in divergence form given by
\begin{equation*}
  \mathcal Au:=-\sum_{k=1}^N\frac{\partial}{\partial
    x_k}\Bigl(\sum_{j=1}^Na_{jk}\frac{\partial u}{\partial
    x_j}+a_ku\Bigr)+\sum_{k=1}^Nb_k\frac{\partial u}{\partial x_k}+cu
\end{equation*}
with real valued $a_{jk}, a_k, b_k, c\in L^\infty(\Omega)$ for $1\leq j,k\leq N$. Furthermore assume that there exists $\alpha>0$ such that
\begin{equation*}
  \alpha|\xi|^2
  \leq \sum_{k=1}^N\sum_{j=1}^Na_{kj}(x)\xi_j\xi_k
\end{equation*}
for all $\xi=(\xi_1,\dots,\xi_N)\in\mathbb R^N$ and almost all $x\in\Omega$. We further assume that the boundary operator is given by
\begin{equation*}
  \mathcal Bu=
  \begin{cases}
    u|_{\Gamma_0}                & \text{on $\Gamma_0$ (Dirichlet)}        \\
    \sum_{k=1}^N\Bigl(\sum_{j=1}^Na_{jk}\frac{\partial u}{\partial
    x_j}+a_ku\Bigr)\nu_k+\beta u & \text{on $\Gamma_1$ (Neumann or Robin)}
  \end{cases}
\end{equation*}
where $\partial\Omega=\Gamma_0\cup\Gamma_1$ is a disjoint union,
$\nu=(\nu_1,\dots,\nu_N)$ the outer unit normal and
$\beta\in L^\infty(\partial\Omega,\mathbb R)$. These boundary conditions
have to be interpreted in a generalised sense via the from associated with $(\mathcal A,\mathcal B)$. That form is given by
\begin{equation*}
  \aaa(u,v):=
  \int_{\Omega}\sum_{k=1}^N\Bigl(\sum_{j=1}^Na_{jk}\frac{\partial u}{\partial
    x_j}+a_ku\Bigr)\overline{\frac{\partial v}{\partial
      x_k}}+\Bigl(\sum_{k=1}^Nb_k\frac{\partial u}{\partial
    x_k}+cu\Bigr)\overline v\,dx
  +\int_{\Gamma_1}\beta u\overline v\,d\sigma
\end{equation*}
for all $u,v$ in a suitable closed subspace of $V$ of $H^1(\Omega)$ such that $H_{0}^1(\Omega)\subseteq V$ depending on the regularity of $\Omega$ and the boundary conditions. Here, $d\sigma$ is the $(N-1)$-dimensional Hausdorff measure restricted to $\Gamma_1$. We will give examples below that guarantee that $\aaa\colon V\times V\to\mathbb R$ is continuous and elliptic, that is, there exists $\alpha_0>0$, $\omega\in\mathbb R$ and $c>0$ with
\begin{equation}
  \label{eq:elliptic-forms}
  \alpha_0\|u\|_V^2\leq\aaa(u,u)+\omega\|u\|_2^2
  \quad\text{and}\quad
  |\aaa(u,v)|\leq c\|u\|_V\|v\|_V
\end{equation}
for all $u,v\in V$. We define $\tilde A\in\mathcal L(V,V')$ by
\begin{equation*}
  \langle\tilde Au,v\rangle_{V',V}:=\aaa(u,v)
\end{equation*}
for all $u,v\in V$ and let $A$ be the part of $\tilde A$ in $L^2(\Omega)$, that is,
\begin{equation*}
  D(A):=\{u\in V\colon \tilde Au\in L^2(\Omega)\}\qquad
  Au:=\tilde Au.
\end{equation*}
Then $-A$ generates a positive irreducible holomorphic $C_0$-semigroup $(T(t))_{t\geq 0}$ on $L^2(\Omega)$; see for instance \cite[Section~1.4, Theorem~2.7 and Corollary~2.11]{ouhabaz:05:ahe}. We recall a connection of the abstract theory with the usual theory of (local) weak solutions to parabolic equations as for instance defined in \cite{aronson:68:nsl,aronson:67:lbs}.

\begin{proposition}
  \label{prop:form-generation}
  Under the above assumptions the following assertions hold.
  \begin{enumerate}[\upshape (i)]
  \item For every $u_0\in L^2(\Omega)$, the function $u(t):=T(t)u_0$ is a weak solution of
    \begin{equation}
      \label{eq:local-weak-solution}
      \begin{aligned}
        \frac{\partial u}{\partial t}+\mathcal Au & =0
                                                  &      & \text{in $\Omega\times (0,\infty)$,} \\
        u(\cdot\,,0)                              & =u_0
                                                  &      & \text{in $\Omega$.}
      \end{aligned}
    \end{equation}
  \item $[T(t)\boldsymbol 1](x)\to 1$ as $t\downarrow 0$ for all
    $x\in\Omega$.
  \end{enumerate}
\end{proposition}
\begin{proof}
  (i) Fix $u_0\in L^2(\Omega)$ and $v\in C_c^\infty(\Omega\times [0,\infty))$. If we set $u(t):=T(t)u_0$, then $u\in C^1((0,\infty),V)\cap C([0,\infty),L^2(\Omega))$. By the definition of $\aaa(\cdot\,,\cdot)$ and an integration by parts we have
  \begin{multline}
    \label{eq:weak-identity-s}
    0=\int_s^\infty\langle \dot u(t)+Au(t),v(t)\rangle\,dt\\
    =-\langle u(s),v(0)\rangle -\int_s^\infty\langle u(t),\dot
    v(t)\rangle\,dt +\int_s^\infty\aaa(u(t),v(t))\,dt
  \end{multline}
  for all $s>0$. Replacing $\mathcal A$ by $\mathcal A+\omega I$ we can assume without loss of generality that $\aaa$ is coercive, that is, \eqref{eq:elliptic-forms} holds with $\omega=0$. Then $\|T(t)\|_{\mathcal L(H}\leq 1$ and since $T(t)$ generates a holomorphic semigroup there there exists a constant $C>0$ such that $\|AT(t)\|_{\mathcal L(H)}\leq Ct^{-1}$ for all $t>0$. Therefore, using that $T(t)u_0\in D(A)$ for all $t>0$ we have from the definition of $A$
  \begin{multline*}
    \alpha\|T(t)u_0\|_V^2
    \leq \aaa(T(t)u_0,T(t)u_0)
    =\langle AT(t)u_0,T(t)u_0\rangle\\
    \leq \|AT(t)u_0\|_H\|T(t)u_0\|_H
    \leq Ct^{-1}\|u_0\|_H^2.
  \end{multline*}
  Hence $\|u(t)\|_V\leq C_1t^{-1/2}\|u_0\|_H$ with $C_1:=\sqrt{C/\alpha}$ for all $t\in(0,1]$. Thus
  \begin{equation*}
    |\aaa(u(t),v(t))|
    \leq M\|u(t)\|_V\|v(t)|_V
    \leq \frac{MC_1}{t^{1/2}}\|u_0\|_H\|v(t)\|_V
  \end{equation*}
  for all $t>0$ and $t\mapsto\aaa(u(t),v(t))$ has an integrable singularity at $t=0$. Letting $s\to 0$ in \eqref{eq:weak-identity-s}we conclude that
  \begin{equation*}
    0 = -\langle u_0,v(0)\rangle
    -\int_0^\infty\langle u(t),\dot v(t)\rangle\,dt
    +\int_0^\infty\aaa(u(t),v(t))\,dt
  \end{equation*}
  for all $v\in C_c^\infty(\Omega\times[0,\infty))$, so by definition $u$ is a weak solution of \eqref{eq:local-weak-solution} as claimed.

  (ii) We define
  \begin{equation*}
    \tilde{\mathcal A}(t)u:=
    \begin{cases}
      \mathcal Au & \text{if $t\geq 0$,} \\
      -\Delta u   & \text{if $t<0$.}
    \end{cases}
  \end{equation*}
  The \emph{extension principle} from \cite[page~621]{aronson:68:nsl} asserts that the function
  \begin{equation*}
    \tilde u(t):=
    \begin{cases}
      T(t)\boldsymbol 1 & \text{if $t>0$,}     \\
      \boldsymbol 1     & \text{if $t\leq 0$,}
    \end{cases}
  \end{equation*}
  is a weak solution of
  \begin{equation*}
    \frac{\partial\tilde u}{\partial t}+\tilde{\mathcal A}(t)\tilde u=0
    \qquad\text{in $\Omega\times\mathbb R$.}
  \end{equation*}
  As the equation is homogeneous, \cite[Theorem~B]{aronson:68:nsl} shows that $\tilde u$ is locally bounded on $\Omega\times\mathbb R$. Hence by \cite[Theorem~C]{aronson:68:nsl} it is locally H\"older continuous on $\Omega\times\mathbb R$.
\end{proof}

We now discuss some representative cases involving a variety of boundary conditions. Note that any type of mixed conditions could be used as well in quite obvious ways.

\begin{example}[Dirichlet Problem]
  \label{ex:dirichlet-general}
  For the Dirichlet problem we choose $\Gamma_0=\partial\Omega$ and $V=H_0^1(\Omega)$. By Proposition~\ref{prop:form-generation} the operator associated with $\aaa$ generates a positive, irreducible holomorphic $C_0$-semigroup $(T(t))_{t\geq 0}$ on $L^2(\Omega)$. By part (ii) of that proposition the semigroup satisfies condition \eqref{eq:zero-regular}. It follows for instance from \cite[Theorem~4.1 and~5.2]{daners:00:hke} or the results in \cite{aronson:68:nsl,arendt:97:ges,arendt:04:see} that $T(t)\in\mathcal L(L^2,L^\infty)$. Hence $T(t))_{t\geq 0}$ is ultra-contractive and $u(t)=T(t)u_0$ is a locally bounded weak solution in $\Omega\times(0,\infty)$. Hence by \cite[Theorem~C]{aronson:68:nsl} it is locally H\"older continuous on $\Omega\times\mathbb R$, which in particular means that $T(t)u_0\in C(\Omega)$ for all $t>0$. This implies that $(T(t))_{t\geq 0}$ satisfies \eqref{eq:smoothing-semigroup}, and thus it is locally smoothing.
\end{example}

\begin{example}[Neumann problems]
  \label{ex:neumann-general}
  We set $\Gamma_1=\partial\Omega$, $\beta=0$ and let $V=H^1(\Omega)$. We further assume that $\Omega$ satisfies a uniform interior cone condition or that it has the $L^1$-$W^{1,1}$-extension property. This is for instance the case for a Lipschitz or a smooth domain. It follows for instance from \cite[Theorem~4.3, 4.3 and~5.2]{daners:00:hke} that $T(t)\in\mathcal L(L^2,L^\infty)$, that is $(T(t))_{t\geq 0}$ is ultra-contractive. It follows as in Example~\ref{ex:dirichlet-general} that $(T(t))_{t>0}$ is a positive irreducible holomorphic $C_0$-semigroup that is locally smoothing.
\end{example}

\begin{example}[Robin problems]
  \label{ex:robin-general}
  Assume that $\Omega$ is a Lipschitz domain, $\Gamma_1=\partial\Omega$ and $\beta\in L^\infty(\partial\Omega)$. According to \cite[Section~3]{daners:09:ipg} the operators $(\mathcal A,\mathcal B)$ can be written in equivalent form with a new $\beta$ that is positive. Then \eqref{eq:elliptic-forms} holds for $V:=H^1(\Omega)$. Now similar arguments as in the Neumann case from Example~\ref{ex:neumann-general} imply that $(T(t))_{t>0}$ is a positive irreducible holomorphic $C_0$-semigroup that is locally smoothing.  In fact, as shown in \cite[Theorem~4.5]{arendt:20:spp} one has even more information, namely strict positivity of $T(t)u_0$ on $\bar\Omega$ for every $u_0\in L^2(\Omega)$ and $t>0$.
\end{example}

\paragraph{Acknowledgement:} Part of this research during visits of WA
to the University of Sydney and of DD to the University of Ulm. Both authors thank for the pleasant stay and the financial support.

\pdfbookmark[1]{\refname}{biblio}

\end{document}